\documentclass[11pt, twoside]{amsart}

\title[Nondegenerate module categories]{Nondegenerate  module categories}

\author{Chelsea Walton}
\address{Department of Mathematics, Rice University,
P.O. Box 1892, Houston, TX 77005-1892, USA}
\email{notlaw@rice.edu}

\author{Harshit Yadav}
\address{Department of Mathematics, University of Alberta,
Edmonton, Alberta, T6G 2G1, Canada}
\email{hyadav3@ualberta.ca}

\usepackage{mathabx}
\usepackage{import}
\usepackage{amsmath}
\usepackage{amsfonts, stmaryrd}
\usepackage{amsthm}
\usepackage{amssymb,bbm}
\usepackage{dsfont}
\usepackage[alphabetic, initials, nobysame]{amsrefs}
\usepackage[english]{babel}
\usepackage{url}
\usepackage{fancyhdr}
\usepackage{graphicx}
\usepackage{verbatim}
\usepackage[normalem]{ulem}
\usepackage[shortlabels]{enumitem}
\newcommand{\stkout}[1]{\ifmmode\text{\sout{\ensuremath{#1}}}\else\sout{#1}\fi}
\usepackage[
colorlinks=true,
linkcolor=black, 
anchorcolor=black,
citecolor=black,
urlcolor=black, 
]{hyperref}
\usepackage[all,cmtip]{xy}
\usepackage{geometry}
\usepackage{microtype} 
\usepackage{mathtools} 
\usepackage[dvipsnames]{xcolor}
\usepackage{tikz-cd,quiver}

\allowdisplaybreaks

\input xy
\xyoption{all}

\definecolor{forest}{rgb}{0.0, 0.5, 0.0}

\newcommand{\bijar}[1][]{%
 \ar[#1]
 \ar@<0.7ex>@{}[#1]|-*=0[@]{\sim}} 

\usepackage{enumitem}

\usepackage{calrsfs}
\DeclareMathAlphabet{\cal}{OMS}{zplm}{m}{n}

\DeclareMathAlphabet{\mathsf}{OT1}{cmss}{m}{n} 

\usepackage[justification=centering]{caption}

\newcommand{\oop}{\mathrm{op}}

\newcommand{\act}{\triangleright}
\DeclareFontFamily{U}{stixscr}{}
\DeclareFontShape{U}{stixscr}{m}{n}{<-> s*[0.8] stix-mathscr}{}

\makeatletter
\newcommand*{\doublerightarrow}[2]{\mathrel{
  \settowidth{\@tempdima}{$\scriptstyle#1$}
  \settowidth{\@tempdimb}{$\scriptstyle#2$}
  \ifdim\@tempdimb>\@tempdima \@tempdima=\@tempdimb\fi
  \mathop{\vcenter{
    \offinterlineskip\ialign{\hbox to\dimexpr\@tempdima+1em{##}\cr
    \rightarrowfill\cr\noalign{\kern.5ex}
    \rightarrowfill\cr}}}\limits^{\!#1}_{\!#2}}}

\newcommand{\coev}{\textnormal{coev}}

\newcommand{\ev}{\textnormal{ev}}
\newcommand{\End}{\operatorname{End}}

\newcommand{\Hom}{\operatorname{Hom}}
\newcommand{\iHom}{\underline{\mathrm{Hom}}}

\newcommand{\id}{\textnormal{id}}

\newcommand{\isomorph}{\stackrel{\sim}{\to}}

\newcommand{\one}{\mathds{1}}

\newcommand{\reg}{\operatorname{reg}}

\newcommand{\Fun}{\mathsf{Fun}_\Bbbk}

\newcommand{\Mod}{\mathsf{Mod}}

\newcommand{\FdVec}{\mathsf{FdVec}}

\newcommand{\cA}{\cal{A}}
\newcommand{\cB}{\cal{B}}
\newcommand{\cC}{\cal{C}}
\newcommand{\cD}{\cal{D}}
\newcommand{\cE}{\cal{E}}

\newcommand{\cM}{\cal{M}}
\newcommand{\cN}{\cal{N}}

\newcommand{\cS}{\cal{S}}
\newcommand{\cT}{\cal{T}}

\newcommand{\cZ}{\cal{Z}}

\def\tE{\mathtt{E}}

\def\tEC{\tE_{\cC}}
\def\tEM{\tE_{\cM}}

\newcommand{\kk}{\Bbbk}

\def\Hmod{H$-$\mathsf{FdMod}}
\def\Bmod{B$-$\mathsf{FdMod}}
\def\ra{\mathrm{ra}}
\def\Fun{\mathsf{Fun}}

\newtheoremstyle{defstyle}
  {0.5cm}                   
  {0.5cm}                   
  {\normalfont}           
  {}     
  {\normalfont\bfseries}  
  {:}                     
  {0.3cm}              
  {\thmname{#1}\thmnumber{ #2}\thmnote{ (#3)}}

\numberwithin{equation}{section}

\newtheorem*{rep@theorem}{\rep@title}
\newcommand{\newreptheorem}[2]{%
\newenvironment{rep#1}[1]{%
 \def\rep@title{#2 \ref{##1}}%
 \begin{rep@theorem}}%
 {\end{rep@theorem}}}
\makeatother

\newtheorem{theorem}{Theorem}[section]

\newtheorem{proposition}[theorem]{Proposition}

\newreptheorem{proposition}{Proposition}
\newtheorem{corollary}[theorem]{Corollary}
\newreptheorem{corollary}{Corollary}
\newtheorem{lemma}[theorem]{Lemma}

\newtheorem{theorem*}{Theorem}
\newreptheorem{theorem}{Theorem}

\theoremstyle{definition}
\newtheorem{definition}[theorem]{Definition}
\newtheorem{notation}[theorem]{Notation}

\newtheorem{hypothesis}[theorem]{Hypothesis}

\newtheorem{example}[theorem]{Example}
\newtheorem{remark}[theorem]{Remark}

\newtheorem{question}[theorem]{Question}

\makeatletter              
\let\c@equation\c@theorem  
\makeatother
\numberwithin{equation}{section}

\makeatletter
\@namedef{subjclassname@2020}{%
  \textup{2020} Mathematics Subject Classification}

\makeatother


\begin{document}

\begin{abstract}
Due to the work of Shimizu (2019),  various  nondegeneracy conditions for braided finite tensor categories are equivalent. This theory is partially extended to braided module categories here. We introduce when a braided module category is  ``nondegenerate"  and   ``factorizable", and establish that these properties are equivalent. The proof involves a new monadicity result for module categories. 
Lastly, we examine the Hopf case, using Kolb's (2020) notion of a quasitriangular comodule algebra to introduce ``factorizable" comodule algebras. We then show that the representation category of a quasitriangular comodule algebra is nondegenerate in our sense precisely when the comodule algebra is factorizable. 
Several examples are provided.
\end{abstract}

\maketitle

\setcounter{tocdepth}{1}

\begin{changemargin}{2.1cm}{2.7cm} 
{\small \tableofcontents}
\end{changemargin}

\section{Introduction}
Braided tensor categories have become ubiquitous objects appearing in mathematical study of $3$d-Topological Quantum Field Theories (3d-TQFTs)\cite{kerler2001non, turaev2017monoidal} and $2$d-Conformal Field Theories (2d-CFTs) \cite{segal1989cft, bakalov2001book, fuchs2002tft}. This, in turn, leads to their importance in understanding invariants of $3$-manifolds, (gapped) topological phases of matter, and topological quantum computation; see, e.g., \cite{wang2010tqc, wen2017zoo}. In all of these applications, one is interested in braided categories with the extra property of \textit{nondegeneracy}, as reviewed below.
On the other hand, as is usual in algebra, (braided) tensor categories are understood through their module categories.
In recent works, the notion of a braided module category over a braided tensor category has been introduced \cite{enriquez2007quasi,brochier2013cyclotomic,kolb2020braided}. These objects have recently played a role in vital applications, such as in the study of quantum character varieties \cite{ben2018quantum}, in examining quantum symmetric pairs \cite{balagovic2019kmatrix, kolb2020braided}, in the definition of 2-categorical Picard groups and in the extension theory of braided finite tensor categories \cite{davydov2021braided}, 
and the resolution of the minimal modular extension conjecture \cite{johnson2024minimal}. Given the importance of nondegeneracy for braided tensor categories, the goal of this article is to study nondegeneracy of braided module categories. 

\smallskip

In this work, we focus on the case when  tensor categories and module categories over them  are finite. Moreover, all linear structures are over an algebraically closed field $\kk$.

\smallskip

Let us review the  notions of nondegeneracy for braided finite tensor categories  $\cC:=(\cC, \otimes, \one, c)$. Namely, $(\cC, \otimes, \one)$ is a tensor category over $\kk$, that comes equipped with a natural isomorphism  (a braiding) $c:=\{c_{X,Y}: X \otimes Y \overset{\sim}{\to} Y \otimes X\}_{X,Y \in \cC}$ satisfying the hexagon axioms. Note that when $\cC$ is finite, the coend object ${\tt C}_{\cC}:= \textstyle \int^{X \in \cC} X^* \otimes X$ of $\cC$ exists and is a Hopf algebra in $\cC$. We say that: 

\smallskip

\begin{itemize}
\item 
$\cC$ is {\it nondegenerate} if, for the Hopf algebra  ${\tt C}_{\cC}$ above, the Hopf pairing $\omega_\cC: {\tt C}_{\cC} \otimes {\tt C}_{\cC} \to \one$ of $ {\tt C}_{\cC}$ \linebreak
yields an isomorphism $\theta_\cC: (\omega_\cC \otimes \id)(\id \otimes \coev_{{\tt C}_{\cC}}): {\tt C}_{\cC} \to {\tt C}_{\cC}^*$ in $\cC$;

\smallskip

\item 
$\cC$ is {\it weakly factorizable} if the map $\Hom_{\cC} (\one, {\tt C}_{\cC}) \to \Hom_{\cC}({\tt C}_{\cC}, \one)$, $f \mapsto \omega_{\cC}(f \otimes \id)$ is bijective;

\smallskip

\item $\cC$ is {\it factorizable} if the embeddings of $\cC$ and  its mirror $\overline{\cC}$ to the Drinfeld $\cZ(\cC)$ yields an equivalence between the Deligne product $\cC \boxtimes \overline{\cC}$ and $\cZ(\cC)$, as braided finite tensor categories; 

\smallskip

\item $\cC$ has {\it trivial symmetric (M\"{u}ger) center} if the full subcategory $\cZ_2(\cC)$ of objects $X$ of $\cC$ for which $c_{Y,X} \circ c_{X,Y} = \id_{X \otimes Y}$, for all $Y \in \cC$, is equivalent to the category of $\kk$-vector spaces.
\end{itemize}

\smallskip

A powerful result of K. Shimizu verifies that these  conditions   are equivalent in the finite case.

\medskip

\noindent {\bf Shimizu's Theorem} \cite[Theorem~1.1]{shimizu2019non} Nondegeneracy, weak factorizability, factorizability, and  trivial symmetric center are equivalent properties for a braided finite tensor category.\qed

\medskip

This builds on several prior works in the semisimple (fusion) case \cite{bruguieres2000premod, muger2003subfactors, drinfeld2010braided, etingof2015tensor} and in Hopf case \cite{reshetikhin1988quantum, radford1994kauffman, takeuchi2001modular}.  In this article, we will introduce a module-theoretic version of nondegeneracy and factorizability, and establish that these properties are equivalent. Analogues of weak factorizability and trivial symmetric center will also be discussed.

\smallskip

Fix $\cC:=(\cC, \otimes, \one, c)$ to be a braided finite tensor category for the rest of the section, unless stated otherwise. A left $\cC$-module category $(\cM, \triangleright)$ is called {\it braided} if it is equipped with a natural isomorphism   $e:=\{e_{X,M}: X \triangleright M \overset{\sim}{\to} X \triangleright M\}_{X \in \cC, M \in \cM}$   satisfying certain compatibility axioms. Examples include the regular module categories $(\cM, \triangleright) = (\cC, \otimes)$ with $e_{X,M}:=c_{M,X} \circ c_{X,M}$. 

\smallskip

Key examples of braided left $\cC$-module categories are the {\it reflective centers} $\cE_\cC(\cM)$ of arbitrary left $\cC$-module categories $\cM$, which were introduced by Laugwitz, Yakimov, and the first author in \cite{laugwitz2023reflective}. Objects are pairs $(M, e^M)$, where  $M \in \cM$ and $e^M:=\{e_{X}^M: X \triangleright M \overset{\sim}{\to} X \triangleright M\}_{X \in \cC}$ is a certain natural isomorphism in $\cM$. Their construction is analogous to the construction of the braided finite tensor category, the Drinfeld center $\cZ(\cC)$, of an arbitrary finite tensor category $\cC$.

\smallskip

Now for a  left $\cC$-module category $(\cM, \triangleright)$, that is not necessarily braided, the right adjoint of $(- \triangleright M): \cC \to \cM$ exists, for each $M \in \cM$. It is denoted by $ \underline{\Hom}(M,-): \cM \to \cC$. Further, when $\cM$ is finite,  the end object ${\tt E}_{\cM}:= \textstyle \int_{M \in \cM} \underline{\Hom}(M,M)$ of $\cM$ exists and is an algebra in $\cC$. For instance, when $\cM = \cC$ as above, we get the  end object ${\tt E}_{\cC}:= \textstyle \int_{X \in \cC} X \otimes X^*$ of $\cC$, which is dual to~${\tt C}_{\cC}$. 
Finally,  using the universal properties of the end objects ${\tt E}_{\cM}$ and ${\tt E}_{\cC}$, we derive a morphism 
$$\omega_\cM: \one \to {\tt E}_{\cC} \otimes {\tt E}_{\cM} \quad (\text{{\it universal copairing}})$$ in $\cC$, that is compatible with the braiding of $\cM$ [Definition~\ref{def:omegaM}].

\smallskip

With these notions, we introduce the concept of nondegeneracy and factorizability for braided module categories. 
From now on, assume that  braided  left $\cC$-module categories $(\cM, \triangleright, e)$ satisfy the conditions specified in Hypothesis~\ref{hyp:Sec4}: nonzero, indecomposable, exact, and finite.

\begin{definition} [Definition~\ref{def:Mnondeg}] For a braided left $\cC$-module category $(\cM, \triangleright, e)$, take the morphism
$$\theta_\cM: (\ev_{{\tt E}_{\cC}} \otimes \id) (\id \otimes \omega_\cM): {\tt E}_{\cC}^* \longrightarrow {\tt E}_{\cM}$$
in $\cC$. We say that $(\cM, \triangleright, e)$ is {\it nondegenerate} if $\theta_\cM$ is an isomorphism in $\cC$.
\end{definition}

\begin{definition} [Definition~\ref{def:Mfact}] For a braided left $\cC$-module category $(\cM, \triangleright, e)$, we have a functor 
$$G_{\cM}: \cM \boxtimes \mathsf{Fun}_{\cC|}(\cM,\cM) \to \cE_{\cC}(\cM), \quad M \boxtimes (F,s) \mapsto (F(M), e^M),$$
where $\mathsf{Fun}_{\cC|}(\cM,\cM)$ is the category of exact $\cC$-module endofunctors $(F,s)$ of $\cM$, and $e^M$ is in terms of $s$.  We say that $(\cM, \triangleright, e)$ is {\it factorizable} if $G_{\cM}$ is an equivalence of braided module categories.
\end{definition}

For instance, the regular left $\cC$-module category $(\cC, \; \triangleright:= \otimes, \; e:=c^2)$ is  nondegenerate (resp., factorizable) as a braided module category precisely when $(\cC, \otimes, \one, c)$ is nondegenerate (resp., factorizable) as a braided finite tensor category [Examples~\ref{ex:regular-fact} and \ref{ex:Mreg-nondeg}].

\smallskip

This brings us to the main result of this article.

\begin{theorem}[Theorem~\ref{thm:FactNondeg}] Consider the terminology above. Then, a braided left $\cC$-module category $(\cM, \triangleright, e)$ is nondegenerate if and only if it is factorizable. \qed
\end{theorem}

To prove this result, we first recall that $\theta_\cM$ is an isomorphism if and only if a restriction functor $\textnormal{Res}_{\hbox{\tiny$\theta_{\cM}$}}: {\tt E}_{\cM}$-$\mathsf{Mod}(\cM) \to {\tt E}_{\cC}^*$-$\mathsf{Mod}(\cM)$ is an equivalence; these  are Eilenberg-Moore categories [$\S$\ref{ss:EMcat}]. We then establish equivalences $H_\cM:  \cM \boxtimes \mathsf{Fun}_{\cC|}(\cM,\cM) \to {\tt E}_{\cM}\text{-}\mathsf{Mod}(\cM)$ [Proposition~\ref{prop:HM}] and  $\widehat{H}_\cC: \cE_\cC(\cM) \to {\tt E}_{\cC}^*\text{-}\mathsf{Mod}(\cM)$ [Proposition~\ref{prop:HC}], and show that the diagram below commutes. 
\[
\xymatrix@C=7pc@R=1.2pc{
\cM \; \boxtimes \;  \mathsf{Fun}_{\cC|}(\cM, \cM)
\ar[r]^(.6){G_{\cM}}
\ar[d]_(.45){H_{\cM}}^{\sim}
&
\cE_{\cC}(\cM)\\
{\tt E}_{\cM}\text{-}\mathsf{Mod}(\cM)
\ar[r]^(.5){\textnormal{Res}_{\hbox{\tiny$\theta_{\cM}$}}}
&
{\tt E}_{\cC}^*\text{-}\mathsf{Mod}(\cM)  \ar[u]_(.5){\widehat{H}_{\cC}}^(.5){\sim}
}
\]
The equivalence $H_\cM$, in particular, involves the  monadicity result for module categories below. 

\begin{proposition}[Proposition~\ref{prop:monad}]
Consider the category of exact endofunctors $\mathsf{Fun}(\cM, \cM)$ of~$\cM$, and the functor $\rho: \cC \to \mathsf{Fun}(\cM, \cM)$ given by $X \to (X \triangleright -)$. Then, the right adjoint $\rho^{\textnormal{ra}}$ of $\rho$ exists, and the adjunction is monadic. Further, this yields an equivalence of categories: \[
\mathsf{Fun}(\cM, \cM) \simeq \mathsf{Mod}\text{-}{\tt E}_{\cM}(\cC).\]

\vspace{-.3in}

 \qed
\end{proposition}

\pagebreak

Returning to Shimizu's Theorem above, we introduce versions of weak-factorizability and trivial symmetric center for braided module categories [Definitions~\ref{def:wf},~\ref{def:trivM}], and obtain the result below.

\begin{theorem} [Lemma~\ref{lem:non-wf}, Theorem~\ref{thm:wf-triv}] \label{thm:wf-intro}
If a braided module category $\cM$ is nondegenerate, then $\cM$ is weakly-factorizable, and as a result, has trivial symmetric center. \qed
\end{theorem}

See $\S\S$\ref{ss:wf}, \ref{ss:triv} for details about this result; we also inquire  when the converse directions hold.

\smallskip

Next, we turn our attention to the Hopf case. Just as finite-dimensional representations of a finite-dimensional quasitriangular Hopf algebra $(H,R)$ yields a braided finite tensor category,
\[
\cC:=H\text{-}\mathsf{FdMod},
\]
the finite-dimensional representations of a finite-dimensional quasitriangular left $H$-comodule algebra $(B,K)$ yields a braided finite left $\cC$-module category,
\[
\cM:=B\text{-}\mathsf{FdMod}.
\]
Here, $R$ (resp., $K$) is referred to as the $R$-matrix (resp., $K$-matrix) of $H$ (resp., $B$). Quasitriangular structures on comodule algebras were introduced by Kolb \cite{kolb2020braided}, and are reviewed in $\S$\ref{sec:qtcomodalg}. It is also well-known that the braiding for $H\text{-}\mathsf{FdMod}$  is nondegenerate if and only if the quasitriangular Hopf algebra $(H,R)$ is factorizable [$\S$\ref{subsec:fact-Hopf}]. 

\smallskip

Now given a finite-dimensional quasitriangular left $H$-comodule algebra $(B,K)$, we introduce the concept of when it is  {\it factorizable}  [Definition~\ref{def:thetaB}]. For example, the coregular quasitriangular left $H$-comodule algebra $(H, R_{21}R)$ is factorizable precisely when the quasitriangular Hopf algebra $(H,R)$ is factorizable [Example~\ref{ex:fact-reg}]. We also establish the result below.

\begin{theorem}[Theorem~\ref{thm:nondeg-comodules}]
Let $H$ be a finite-dimensional quasitriangular Hopf algebra. Then, a finite-dimensional quasitriangular left $H$-comodule algebra $B$ is factorizable if and only if $B\text{-}\mathsf{FdMod}$ is a nondegenerate left $(H\text{-}\mathsf{FdMod})$-module category. \qed
\end{theorem}

One supply of examples of factorizable comodule algebras pertain to the reflective centers mentioned above. The reflective center in the Hopf case, $\cE_{H\text{-}\mathsf{FdMod}}(A\text{-}\mathsf{FdMod})$, for $A$ a left $H$-comodule algebra (not necessarily braided), is equivalent to a category, $R_H(A)\text{-}\mathsf{FdMod}$  \cite[Theorem~6.6]{laugwitz2023reflective}. Here, $R_H(A)$ is a finite-dimensional quasitriangular left $H$-comodule algebra, called a {\it reflective algebra}. This yields the family of factorizable comodule algebras, as presented below.

\begin{proposition}[Proposition~\ref{prop:RHA-fact}]
Let $H$ be a finite-dimensional quasitriangular Hopf algebra. 
Then, the reflective algebra $R_H(A)$ is a factorizable left $H$-comodule algebra when it is $H$-simple. 
\qed
\end{proposition}

\smallskip

Some questions for further investigation are included here; see, for instance, Questions~\ref{ques:JFR} and~\ref{ques:conv}.
In general, we anticipate that the notions of nondegeneracy here will play a role in the numerous applications of braided module categories mentioned at the beginning of the section.


\section{Preliminaries on monoidal categories and their braidings}

Throughout our work, all linear structures will be over an algebraically closed field $\Bbbk$. In this section, we provide background material on categories [$\S$\ref{ss:cats}], on co(end)s for categories [$\S$\ref{ss:coend}], on monoidal categories [$\S$\ref{ss:monoidal}], on braided finite tensor categories [$\S$\ref{ss:btc}], on coend Hopf algebras in braided finite tensor categories [$\S$\ref{ss:coendBTC}], and on nondegenerate tensor categories [$\S$\ref{subsec:nondeg-monoidal}].

\subsection{Preliminaries on categories}  \label{ss:cats}
See \cite{mitchell1965categories, mac1998categories,walton2024} for a review of $\kk$-linear abelian  categories, and refer to \cite[$\S\S$1.8, 1.11]{etingof2015tensor} for details of the material below.

\subsubsection{Deligne product of  categories} Let $\cA$, $\cA'$ be $\Bbbk$-linear abelian categories. The {\it Deligne product} of $\cA$ and $\cA'$ is a $\kk$-linear abelian category $\cA \boxtimes \cA'$ endowed with a functor $\boxtimes : \cA \times \cA' \rightarrow \cA\boxtimes \cA'$ that is $\Bbbk$-linear and right exact in each variable, and is universal among such functors out of $\cA \times \cA'$. When the Deligne product exists, we have the following natural isomorphism:
\begin{equation*} \label{eq:deligneHom}
\Hom_{\cA  \boxtimes \cA'}(X \boxtimes X', Y \boxtimes Y') \; \cong \; \Hom_{\cA}(X, Y) \otimes_\kk \Hom_{\cA'}(X', Y').
\end{equation*}
for $X,Y \in \cA$ and $X',Y' \in \cA'$.

\subsubsection{Finiteness}  A $\Bbbk$-linear abelian category $\cA$ is {\it locally finite} if $\Hom_{\cA}(V,W)$ is a finite-dimensional $\Bbbk$-vector space for each $V,W \in \cA$, and every object has a finite filtration by simple objects. 

\smallskip

A locally finite category $\cA$ is {\it finite} if there are enough projectives and finitely many isomorphism classes of simple objects. Equivalently, a $\kk$-linear category $\cA$ is finite if it is equivalent to the category of finite-dimensional modules over a finite-dimensional $\kk$-algebra.
For example, the category $\mathsf{FdVec}$ of finite-dimensional $\kk$-vector spaces is finite and abelian.
Moreover, the Deligne product of finite abelian categories exists and is finite. 

\begin{hypothesis} \label{hyp:cat}
Unless stated otherwise, we assume the following conditions.
\begin{enumerate}[(a)]
\item All categories here are $\kk$-linear and abelian (and will often be finite as specified below). 
\smallskip
\item Given two such categories $\cA$ and $\cB$, any functor $\cA \to \cB$ is assumed to be linear and right exact; the category of such functors is denoted by $\mathsf{Fun}(\cA,\cB)$.
\smallskip
\item When we write $V \otimes_\kk X$, for $V \in \mathsf{FdVec}$ and $X \in \cA$,  we assume that $X$ has an underlying vector space structure via the equivalence of $\cA$ and a category of finite-dimensional modules over a finite-dimensional algebra over $\kk$.
\end{enumerate}
\end{hypothesis}


\subsection{(Co)ends for categories}  \label{ss:coend}
We refer to reader to \cite[$\S\S$IX.4--6]{mac1998categories} for the details here.

\smallskip

Take categories $\cS$ and $\cT$, with functors $F_1, F_2: \cS^{\textnormal{op}} \times \cS \to \cT$. A {\it dinatural transformation} from $F_1$ to $F_2$ is a collection of morphisms $\delta:=\{\delta_{\cS}(S): F_1(S,S) \to F_2(S,S)\}_{S \in \cS}$,
such that  $F_2(S,f) \circ \delta_{\cS}(S) \circ F_1(f,S) = F_2(f,S') \circ \delta_{\cS}(S') \circ F_1(S',f)$ for each morphisms $f: S \to S'$ in $\cS$. We denote this dinatural transformation by $\delta: F_1 \overset{\hbox{\Large $..$}}{\to} F_2$.

\smallskip

For $T \in \cT$, let $K_T: \cS^{\textnormal{op}} \times \cS \to \cT$ be the {\it constant functor} with $K_T(S, S') =: T$, for all $S,S' \in \cS$.

\smallskip

The {\it end} of a functor $F: \cS^{\textnormal{op}} \times \cS \to \cT$ is an object $\mathtt{E} \in \cT$ equipped with a dinatural transformation $\pi: K_{\mathtt{E}} \overset{\hbox{\Large $..$}}{\to} F$, such that for any $E \in \cT$ and any  $\delta: K_E \overset{\hbox{\Large $..$}}{\to} F$, there is a unique morphism $f: E \to \mathtt{E}$  with $\pi_{\cS}(S) \circ f = \delta_{\cS}(S)$ for all $S \in \cS$. We denote $(\mathtt{E}, \pi)$ by $\textstyle \int_{S \in \cS} F(S,S)$.

\smallskip

The {\it coend} of a functor $F: \cS^{\textnormal{op}} \times \cS \to \cT$ is an object $\mathtt{C} \in \cT$ equipped with a dinatural transformation $\iota: F \overset{\hbox{\Large $..$}}{\to} K_{\mathtt{C}}$, such that for any $C \in \cT$ and any $\delta: F \overset{\hbox{\Large $..$}}{\to} K_C$, there is a unique morphism $f: \mathtt{C} \to C$  with $f \circ \iota_{\cS}(S) = \delta_{\cS}(S)$ for all $S \in \cS$. We denote $(\mathtt{C}, \iota)$ by $\textstyle \int^{S \in \cS} F(S,S)$.

\smallskip

When (co)ends exist, they commute with each other, and with other (co)limits in $\cT$. Moreover, right (resp., left) adjoints preserve limits (resp., colimits). So, 
$R\left(\textstyle \int_{S \in \cS} F(S,S)\right) \cong   \textstyle \int_{S \in \cS} RF(S,S)$ and  $L\left(\textstyle \int^{S \in \cS} F(S,S)\right) \cong   \textstyle \int^{S \in \cS} LF(S,S)$, for  an adjunction $(L: \cT \to \cT') \dashv (R: \cT' \to \cT)$.

\begin{proposition}\label{prop:coend}  \cite[Lemma~3.2]{shimizu2017relative}
For finite categories $\cA$ and $\cB$, we have an equivalence of categories, with quasi-inverse, respectively:
\[
\begin{array}{ll}
\Phi: \cA^{\textnormal{op}} \boxtimes \cB \to \mathsf{Fun}(\cA,\cB), & 
A \boxtimes B \mapsto \Hom_{\cA}(-,A)^* \otimes_\kk B,\\[.6pc]
\Psi: \mathsf{Fun}(\cA,\cB) \to \cA^{\textnormal{op}} \boxtimes \cB, & 
F \mapsto \textstyle \int_{C \in \cA} C \boxtimes F(C).
\end{array}
\] 
In particular, we get the following identities in $\cA^{\textnormal{op}} \boxtimes \cB$ and in $\mathsf{Fun}(\cA,\cB)$, respectively:
\begin{eqnarray}
A \boxtimes B &\cong& \textstyle \int_{C \in \cA} C \boxtimes (\Hom_{\cA}(C,A)^*  \otimes_\kk B), \label{eq:coend-1}\\[.2pc]
F(-) &\cong & \textstyle \int_{C \in \cA} \Hom_{\cA}(-,C)^*  \otimes_\kk F(C), \label{eq:coend-2}
\end{eqnarray}
for each $A \in \cA^{\textnormal{op}}$, $B \in \cB$, and $F \in\mathsf{Fun}(\cA,\cB)$. 
\qed
\end{proposition}


\subsection{Monoidal categories} \label{ss:monoidal}
See various parts of \cite{etingof2015tensor, walton2024} for details on the material here.

\subsubsection{Monoidal categories} A {\it monoidal category} consists of a category $\cC$ equipped with  a bifunctor $\otimes\colon  \cC \times \cC \to \cC$,  a natural isomorphism $a_{X,Y,Z}\colon  (X \otimes Y) \otimes Z \overset{\sim}{\to} X \otimes (Y \otimes Z)$ for $X,Y,Z \in \cC$, an object $\one \in \cC$, and  natural isomorphisms $l_X\colon  \one \otimes X \overset{\sim}{\to} X$ and $r_X\colon  X  \otimes \one \overset{\sim}{\to} X$ for  $X \in \cC$, 
such that the pentagon and triangle axioms hold. 
Unless stated otherwise, we use MacLane's strictness theorem to assume that all monoidal categories are {\it strict} in the sense that 
$a_{X,Y,Z},\; l_X,\; r_X$ are identity maps.

\smallskip

A {\it lax monoidal functor}  between monoidal categories $(\cC, \otimes, \one)$ and $(\cC', \otimes', \one')$ is a functor \linebreak $F \colon \cC \to \cC'$ equipped with a natural transformation $F_{X,Y}\colon  F(X) \otimes' F(Y) \to F(X \otimes Y)$ for all $X,Y \in \cC$, and  an isomorphism $F_0\colon  \one' \to F(\one)$ in $\cC'$, 
that satisfy the associativity and unitality constraints; it is {\it (strong) monoidal} when $\{F_{X,Y}\}_{X,Y \in \cC}$ and $F_0$ are isomorphisms.
An {\it equivalence (resp., isomorphism) of  monoidal categories} is  a  monoidal functor that yields an equivalence (resp., isomorphism) of the underlying categories; it is denoted by $\overset{\otimes}{\simeq}$ (resp., $\overset{\otimes}{\cong}$).

\smallskip

Given a monoidal category $(\cC, \otimes, \one, a, l, r)$ (in the non-strict case for clarity), its {\it opposite monoidal category} is defined as  $\cC^{\otimes \textnormal{op}}:=(\cC, \otimes^{\textnormal{op}}, \one, a^{\textnormal{op}}, l^{\textnormal{op}}, r^{\textnormal{op}})$, with $X \otimes^{\textnormal{op}} Y := Y \otimes X$ and $a^{\textnormal{op}} _{X,Y,Z} := a^{-1}_{Z,Y,X}$ and $l_X^{\textnormal{op}} = r_X$  and $r_X^{\textnormal{op}} = l_X$, for all $X,Y,Z \in \cC$.

\subsubsection{Algebraic structures in monoidal categories} An {\it algebra} in a monoidal category $(\cC, \otimes, \one)$ is an object $A \in \cC$, equipped with morphisms, $m_A: A \otimes A \to A$ and $u_A: \one \to A$, subject to associativity and unitality axioms. These form a category, $\mathsf{Alg}(\cC)$, where from morphisms $(A,m_A,u_A)$ to $(A',m_{A'},u_{A'})$ are morphisms $A \to A'$ in $\cC$ compatible with $m_A, m_{A'}$ and with $u_A, u_{A'}$. 

\smallskip

Dually, a {\it coalgebra} in a monoidal category $(\cC, \otimes, \one)$ is a triple $(C, \Delta, \varepsilon)$, where $C \in \cC$, and $\Delta_C: C \to C \otimes C$, $\varepsilon_C: C \to \one$ are morphisms in $\cC$ satisfying coassociativity and counitality axioms. Likewise, the category $\mathsf{Coalg}(\cC)$ can be formed by defining morphisms in a dual fashion.

\smallskip

 Take $A \in \mathsf{Alg}(\cC)$.  A {\it left $A$-module} in $\cC$ is an object $M \in \cC$, equipped with a  morphism \linebreak $\hbox{\scriptsize{$\vartriangleright$}}_M: A \otimes M \to M$, subject to associativity and unitality axioms. These form $A$-$\mathsf{Mod}(\cC)$, a category where morphisms $(M, \hbox{\scriptsize{$\vartriangleright$}}_M) \to (M', \hbox{\scriptsize{$\vartriangleright$}}_{M'})$ are maps $M \to M'$ in $\cC$ compatible with $\hbox{\scriptsize{$\vartriangleright$}}_M$ and $ \hbox{\scriptsize{$\vartriangleright$}}_{M'}$.
 
\smallskip

 Likewise, the category, $\mathsf{Mod}$-$A(\cC)$, consisting of right $A$-modules in~$\cC$ can be defined. There is also a  category, $A$-$\mathsf{Bimod}(\cC)$, of $A$-bimodules in~$\cC$, whose objects are triples $(M, \hbox{\scriptsize{$\vartriangleright$}}_M, \hbox{\scriptsize{$\vartriangleleft$}}_M)$ with $(M, \hbox{\scriptsize{$\vartriangleright$}}_M) \in A$-$\mathsf{Mod}(\cC)$ and $(M, \hbox{\scriptsize{$\vartriangleleft$}}_M) \in \mathsf{Mod}$-$A(\cC)$, subject to a compatibility condition.

\subsubsection{Rigidity} A monoidal category $(\cC, \otimes, \one)$ is {\it rigid} if it comes equipped with left and right dual objects,  i.e., for each $X \in \cC$ there exist, respectively, an object $X^* \in \cC$ with co/evaluation maps $\ev_X^L \colon  X^* \otimes X \to \one$ and $\coev_X^L \colon  \one \to  X \otimes X^*$, and an object ${}^*X \in \cC$ with co/evaluation maps $\ev_X^R \colon  X \otimes {}^*X \to \one$, $\coev_X^R \colon \one \to  {}^*X \otimes X$,  satisfying  coherence conditions.

\smallskip

In a rigid category, duals of morphisms can be formed. For instance, for $f: X \to Y \in \cC$, its left dual defined by $f^*:= (\ev_Y^L \otimes \id_{X^*})(\id_{Y^*} \otimes f \otimes \id_{X^*})(\id_{Y^*} \otimes \coev^L_X): Y^* \to X^*$. Moreover, we get
\begin{equation} \label{eq:dualmor}
\ev_Y^L\circ (\id_{Y^*} \otimes f) = \ev_X^L \circ (f^* \otimes \id_X), \qquad 
\ev_Y^R \circ (f \otimes \id_{{}^* Y}) = \ev_X^R \circ (\id_X \otimes {}^* \hspace{-.02in} f).
\end{equation}

Moreover, for $(A,m,u) \in \mathsf{Alg}(\cC)$, we obtain that $(A^*,m^*,u^*) \in \mathsf{Coalg}(\cC)$. On the other hand, for $(C,\Delta,\varepsilon) \in \mathsf{Coalg}(\cC)$, we obtain that $(C^*,\Delta^*,\varepsilon^*) \in \mathsf{Alg}(\cC)$.

\subsubsection{(Multi-)tensor and (multi-)fusion categories} A  {\it multitensor category} is an abelian, locally finite, rigid, monoidal category $(\cC, \otimes, \one)$ such that $\otimes$ is $\Bbbk$-linear in each slot; it is {\it tensor} if, further, $\End_\cC(\one) \cong \Bbbk$. A  (multi)tensor category is  {\it (multi)fusion} if it is both finite and semisimple.

\smallskip

For instance, $\mathsf{FdVec}$ is a fusion category, with $\otimes := \otimes_\kk$ and $\one := \kk$.
If $\cC$ is a tensor (resp., finite tensor, fusion) category, then is so $\cC^{\otimes \textnormal{op}}$.

\smallskip

 A {\it tensor functor} is a $\Bbbk$-linear, exact, faithful, monoidal functor $F$ between  tensor (or fusion) categories $\cC$ and $\cC'$, with $F(\one) = \one'$.

\subsubsection{Deligne product of various monoidal categories } If $\cC$ and $\cC'$ are two monoidal categories, then the Deligne product of underlying categories $\cC \boxtimes \cC'$, if it exists, is also monoidal. 
Moreover, the Deligne tensor product of two  tensor (resp., finite tensor, fusion) categories exists and is a tensor (resp.,  finite tensor, fusion) category; see \cite[\S4.6]{etingof2015tensor}.


\subsection{Braided finite tensor categories}  \label{ss:btc}
See \cite[$\S\S$8.1--8.5]{etingof2015tensor} and \cite[$\S\S$3.1, 5.1, 6.2]{turaev2017monoidal} for more information on the material below. Let $\cC:=(\cC, \otimes, \one)$ be a finite tensor category here.

\subsubsection{Braided categories} \label{ss:br-cat} A finite tensor category $\cC:=(\cC, \otimes, \one)$ is {\it braided} if it is a equipped with  a natural isomorphism $c:=\{c_{X,Y}\colon  X \otimes Y \overset{\sim}{\to} Y \otimes X\}_{X,Y \in \cC}$ ({\it braiding}), such that the following hexagon axioms hold for each $X,Y,Z \in \cC$: 
\begin{align}
c_{X \otimes Y, Z} &  =  (c_{X,Z} \otimes \id_Y) (\id_X \otimes c_{Y,Z}), \label{eq:braid1} \\[.4pc]
c_{X, Y \otimes Z} & = (\id_Y \otimes c_{X,Z}) (c_{X,Y} \otimes \id_Z).  \label{eq:braid2} 
\end{align}
As a consequence, we obtain that for each $X \in \cC$:
\begin{align}
c_{\one, X} =  c_{X, \one} = \id_X. \label{eq:braid-unit}
\end{align}

We also have a {\it mirror braiding} on $\cC$ given by 
$c^{-1}:=\{c_{Y,X}^{-1}\colon X \otimes Y \overset{\sim}{\to} Y \otimes X\}_{X,Y \in \cC}$. 
We have a corresponding  braided finite tensor category, $\overline{\cC}:=(\cC,  c^{-1})$,
and call it the {\it mirror} of $(\cC, c)$.

\smallskip

A {\it braided tensor functor} between braided finite tensor categories $\cC$ and $\cC'$ is a tensor functor \linebreak $(F, F_{-,-},F_0)\colon  \cC \to \cC'$ such that 
$F_{Y,X} \circ c'_{F(X),F(Y)} = F(c_{X,Y}) \circ  F_{X,Y}$,
for all $X,Y \in \cC$. An {\it equivalence (resp., isomorphism) of  braided finite tensor categories} is a braided tensor functor that yields an equivalence (resp., isomorphism) of the underlying categories. 

\smallskip

Given two braided finite tensor categories $(\cC,c)$ and $(\cC', c')$, the Deligne product of the underlying tensor categories $\cC \boxtimes \cC'$ is also braided.

\subsubsection{Algebraic structures in braided finite tensor categories} 
A {\it bialgebra} in  $(\cC, \otimes, \one, c)$ is a tuple \linebreak $(A,m_A,u_A,\Delta_A, \varepsilon_A)$, where 
$(A,m_A,u_A) \in \mathsf{Alg}(\cC)$, and $(A,\Delta_A,\varepsilon_A) \in \mathsf{Coalg}(\cC)$, with $\Delta_A, \varepsilon_A \in \mathsf{Alg}(\cC)$. Here, $A \otimes A \in \mathsf{Alg}(\cC)$, where $m_{A \otimes A} = (m_A \otimes m_A)(\id_A \otimes c_{A,A} \otimes \id_A)$ and $u_{A \otimes A} := u_A \otimes u_A$.

\smallskip

A {\it Hopf algebra} in  $(\cC, \otimes, \one, c)$ is a tuple $(A,m_A,u_A,\Delta_A, \varepsilon_A, S_A)$, where $(A,m_A,u_A,\Delta_A, \varepsilon_A)$ is a bialgebra in $\cC$, and  $S_A: A \to A$ in $\cC$ satisfying
$m_A(S_A \otimes \id_A)\Delta_A = u_A \varepsilon_A = m_A(\id_A \otimes S_A)\Delta_A$.

\subsubsection{Drinfeld centers}\label{sec:Z(C)}  An important example of a braided finite tensor category is the {\it Drinfeld center}  $\cZ(\cC)$ of a finite  tensor category $\cC$. Its objects are pairs $(V, c^V)$, where $V$ is an object of $\cC$, and $c^V:=\{c^V_X \colon  X \otimes V \overset{\sim}{\to} V \otimes X\}_{X \in \cC}$ is a natural isomorphism ({\it half-braiding}), satisfying the condition
$c^V_{X \otimes Y} = (c^V_X \otimes \id_{Y})  (\id_{X} \otimes c^V_Y)$.

\smallskip

Morphisms $(V, c^V) \to (W, c^W)$ of $\cZ(\cC)$ are  $f \in \Hom_{\cC}(V,W)$  such that $(f \otimes \id_X)  c^V_X = c^W_X  (\id_X \otimes f)$, for all $X \in \cC$. The monoidal product of $\cZ(\cC)$ is $(V,c^V) \otimes (W,c^W):=(V\otimes W, c^{V \otimes W})$,   where we define  $c^{V\otimes W}_X :=  (\id_V \otimes c^W_X)(c^V_X \otimes \id_W)$, for all $X \in \cC$.
The braiding of $\cZ(\cC)$ is defined by 
$c_{(V,c^V),(W,c^W)}:= c^W_V:  (V \otimes W,c^{V \otimes W}) \to (W \otimes V,c^{W \otimes V})$.

\subsubsection{Embeddings into Drinfeld centers in the braided case}
Given a braided finite tensor category $(\cC,c)$, let $\overline{\cC}:=(\cC,c^{-1})$ be its mirror. 
Then, we have the following fully faithful, braided tensor functors from  $\cC$ to its Drinfeld center $\cZ(\cC)$:
\begin{equation} \label{eq:G+G-}
G^+_\cC: \cC \to \cZ(\cC), \; V \mapsto (V, c^V) \quad \quad \quad G^-_\cC: \overline{\cC} \to \cZ(\cC), \; V \mapsto (V, (c^{-1})^V).
\end{equation}


\subsection{Coend Hopf algebras in braided finite tensor categories} \label{ss:coendBTC}
Take the coend object of~$\cC$:
\[ 
\mathtt{C}_{\cC}:= \textstyle \int^{X \in \cC} X^* \otimes X.
\]
By its universal construction, $\mathtt{C}_{\cC}$ is attached to canonical morphisms: 
$\{\iota_{\cC}(X): X^* \otimes X \to  \mathtt{C}_{\cC}\}_{X \in \cC}$.

\begin{lemma}\cite[$\S\S$6.4--6.5]{turaev2017monoidal} \label{lem:EcHopfAlg}
When $\cC$ is a braided finite tensor category, the coend object $\mathtt{C}_{\cC}$ exists. Further, we have the following statements.
\begin{enumerate}[\upshape (a)]
\item  $\mathtt{C}_{\cC}$ is a Hopf algebra in $\cC$. For $X \in \cC$, the operations are given by:
{\small
\[
\begin{array}{rl}
m_{\mathtt{C}_{\cC}} \bigl(\iota_{\cC}(X) \otimes \iota_{\cC}(Y)\bigr) & = \; \iota_{\cC}(X \otimes Y) \; \bigl(c_{X^*,Y^*} \otimes \id_X \otimes \id_Y\bigr) \bigl(\id_{X^*} \otimes c_{X,Y^*} \otimes \id_Y\bigr),\\[.2pc]
u_{\mathtt{C}_{\cC}} & = \;  \iota_{\cC}(\one),\\[.2pc]
\Delta_{\mathtt{C}_{\cC}} \; \iota_{\cC}(X) &= \; \bigl(\iota_{\cC}(X) \otimes \iota_{\cC}(X)\bigr) \bigl(\id_{X^*} \otimes \coev^L_X \otimes \id_X\bigr), \\[.2pc]
\varepsilon_{\mathtt{C}_{\cC}} \; \iota_{\cC}(X) &= \; \ev^L_X,\\[.2pc]
S_{\mathtt{C}_{\cC}} \; \iota_{\cC}(X) &= \bigl(\ev^L_X \otimes \id_{{\tt C}_{\cC}}\bigr)  \bigl(\id_{X^*} \otimes c^{-1}_{{\tt C}_{\cC}, X}\bigr)  \bigl(\id_{X^*} \otimes \iota_{\cC}(X^*) \otimes \id_X\bigr)  \bigl(\coev^L_{X^*} \otimes \id_{X^*} \otimes \id_X\bigr).
\end{array}
\]
}
 
\item  $\mathtt{C}_{\cC}$ comes equipped with a Hopf pairing $\omega_{\cC}: \mathtt{C}_{\cC} \otimes \mathtt{C}_{\cC} \to \one$,
that is defined by the following condition, for each $X,Y \in \cC$:
\begin{equation*} \label{eq:omegaC}
\omega_{\cC} \; \bigl(\iota_{\cC}(X) \otimes \iota_{\cC}(Y)\bigr) 
= (\ev^L_X \otimes \ev^L_Y)  \bigl(\id_{X^*} \otimes (c_{Y^*,X} \circ c_{X,Y^*}) \otimes \id_Y\bigr). 
\end{equation*}
\end{enumerate}

\vspace{-.3in}

\qed
\end{lemma}

Next, consider the following end objects in $\cC$:
\[ 
\mathtt{E}_{\cC}:= \textstyle \int_{X \in \cC} X \otimes X^*.
\]
By universality, $\mathtt{E}_{\cC}$ is attached to canonical morphisms in~$\cC$: 
$\{\pi_{\cC}(X): \; \mathtt{E}_{\cC} \to X \otimes X^*\}_{X \in \cC}$.
Further, 
\begin{align} \label{eq:EcCc}
 {\tt E}_\cC^* \; = \; \textstyle \bigl(\int_{X \in \cC}  X \otimes X^* \bigr)^*
\; \cong \; \textstyle \int^{X \in \cC}  &X^{**} \otimes X^* \; \cong \; \textstyle \int^{X \in \cC} X^{*} \otimes X \; = \;  {\tt C}_\cC,\\[.4pc]
 \label{eq:EcCc2}
(\pi_\cC(X))^* &:= \iota_\cC(X^*).
\end{align}
The right-dual versions of \eqref{eq:EcCc} and \eqref{eq:EcCc2} also hold, labeled as \eqref{eq:EcCc}$'$ and \eqref{eq:EcCc2}$'$, respectively.


\subsection{Nondegenerate tensor categories} 
\label{subsec:nondeg-monoidal}

The main reference for the content here is \cite{shimizu2019non}. See also \cite[\S8.6]{etingof2015tensor} and various parts of 
 \cite{kerler2001non, muger2003subfactors}.
There are several ways that a braided finite tensor category $\cC:=(\cC, \otimes, \one, c)$ could be nondegenerate; these notions are described as follows.

\smallskip

First, the fully faithful functors $G^+_\cC: \cC \to \cZ(\cC)$ and $G^-_\cC: \bar{\cC} \to \cZ(\cC)$ from \eqref{eq:G+G-} can be combined to form a functor: 
\[ G_\cC: \cC \boxtimes \bar{\cC} \to \cZ(\cC), \quad V \boxtimes W \mapsto (V \otimes W, c^{V \otimes W}), \]
where $c^{V \otimes W}_X:= (\id_V \otimes c_{W,X}^{-1})(c_{X,V} \otimes \id_W)$,  for $X \in \cC$.
The functor $G_\cC$ is braided and tensor. If $G_\cC$ is an equivalence (of braided finite tensor categories), then we say that $\cC$ is {\it factorizable}. 

\smallskip

Next, the {\it symmetric center} (or {\it M\"{u}ger center}) of $\cC$, denoted by $\cZ_2(\cC)$,  is defined as the full subcategory of $\cC$ consisting of the objects:
\[
\text{Ob}(\cZ_2(\cC)) := \{X \in \cC ~|~ c_{Y,X}\circ c_{X,Y} = \id_{X \otimes Y}, \; \forall Y \in \cC\}.
\]
We say that $\cC$ has {\it trivial symmetric center} if $\cZ_2(\cC) \simeq \FdVec$.

\smallskip

On the other hand, recall the notation of $\S$\ref{ss:coendBTC}. 
We say that $\cC$ is {\it nondegenerate} if 
\[
\xymatrix@C=5pc{
\theta_{\cC}: \mathtt{C}_{\cC}
\ar[r]^(.43){\id \; \otimes \;\coev^L_{\mathtt{C}_{\cC}}}
&
\mathtt{C}_{\cC} \otimes \mathtt{C}_{\cC} \otimes \mathtt{C}_{\cC}^*
\ar[r]^(.6){\omega_\cC \; \otimes \; \id}
& 
\mathtt{C}_{\cC}^*
} 
\]
is an isomorphism in $\cC$.

\smallskip

Lastly, retaining the notation above, we say that a braided finite tensor category $\cC$ is {\it weakly-factorizable} if the  map below is bijective:
\[
\Omega_{\cC}: \Hom_\cC(\one, \mathtt{C}_{\cC}) \to \Hom_\cC( \mathtt{C}_{\cC}, \one), 
\quad f \mapsto \omega_{\cC} \bigl(f \otimes \id_{\mathtt{C}_{\cC}}\bigr).
\]

\smallskip

Now we have a powerful result of K. Shimizu \cite {shimizu2019non} on the properties above in the finite setting.

\begin{theorem}\label{thm:Shimizu} \cite[Theorem~1.1]{shimizu2019non} The following conditions are equivalent for a braided finite tensor category: \textnormal{(a)} nondegeneracy; \textnormal{(b)} factorizability; \textnormal{(c)} weak factorizability; and \textnormal{(d)} having trivial symmetric center. \qed
\end{theorem} 

This generalizes results in the fusion setting, e.g.,  appearing in \cite{bruguieres2000premod, muger2003subfactors, drinfeld2010braided}.

\smallskip

For example, the Drinfeld center $\cZ(\cC)$ of a (not necessarily braided) finite tensor category $\cC$ is factorizable by \cite[Proposition~8.6.3]{etingof2015tensor}, and thus, it is weakly factorizable, nondegenerate, and has trivial symmetric center by Theorem~\ref{thm:Shimizu}.


\section{Preliminaries on module categories and their braidings} 

In this section, we provide background material on module categories over tensor categories [$\S$\ref{sec:modcat}], on braidings of such module categories [$\S$\ref{sec:brmod}], and on reflective centers, which are key examples of braided module categories [$\S$\ref{sec:refcenter}].


\subsection{Module categories}   \label{sec:modcat}
See various parts of \cite{etingof2015tensor, walton2024} for details on the material below.
Let $\cC$ be a monoidal category unless stated otherwise.

\subsubsection{Module categories} A {\it left $\cC$-module category}  is a category $\cM$ equipped with
a bifunctor $\act: \cC \times \cM \to \cM$,  a natural isomorphism $m:=\{m_{X,Y,M}\colon  (X \otimes Y) \act M \overset{\sim}{\to} X \act (Y \act M)\}_{X,Y \in \cC, M \in \cM}$, and 
 a natural isomorphism $p:=\{p_M\colon  \one \act M \overset{\sim}{\to} M\}_{M \in \cM}$, 
such that  associativity and unitality axioms hold.
Right $\cC$-module categories are defined likewise.

\begin{hypothesis} 
Unless stated otherwise, by the strictness theorem for left module categories, we will assume that all left $\cC$-module categories are {\it strict} in the sense that for  $X, Y \in \cC$ and $M \in \cM$:
\[
X \otimes Y \act M := (X \otimes Y) \act M = X \act (Y \act M), \quad \quad \quad \quad
M := \one \act M;
\]
 that is, $m_{X,Y,M},\; p_M$ are identity maps. We assume the same for right $\cC$-module categories.
\end{hypothesis}

For example, consider the following left $\cC$-module categories.  The {\it regular left $\cC$-module category} is given by
$\cC_{\text{reg}} :=  (\cC, \; \triangleright:= \otimes)$. 
For $A \in \mathsf{Alg}(\cC)$,  we also have that $\mathsf{Mod}$-$A(\cC)$ is a left $\cC$-module category. Namely, $X  \act (M, \hbox{\scriptsize{$\vartriangleleft$}}_M) := \bigl(X \otimes M, \; \id_X \otimes \hbox{\scriptsize{$\vartriangleleft$}}_M\bigr)$, for $X \in \cC$ and $(M, \hbox{\scriptsize{$\vartriangleleft$}}_M) \in \mathsf{Mod}$-$A(\cC)$. Right versions are defined likewise.

\smallskip

A {\it left $\cC$-module functor} between left $\cC$-module categories $(\cM, \act)$ and $(\cM', \act')$ is a functor \linebreak $F: \cM \to \cM'$ equipped with a natural isomorphism,
$s:=\{s_{X,M}: F(X \act M) \isomorph X \act' F(M)\}_{X \in \cC, M \in \cM}$
such that the following associativity and unitality axioms hold:
\begin{eqnarray} 
s_{X \otimes Y,M} &=&   (\id_X \act' s_{Y,M}) \; s_{X, Y \act M},  \label{eq:mod-s-axioms}\\[.2pc]
s_{\one, M}  &=& F(\id_M). \label{eq:mod-s-axioms-2}
\end{eqnarray}

\smallskip

A {\it left module category over a finite tensor category $\cC$} is a left $\cC$-module category $(\cM, \act)$ that is abelian, finite, bilinear on morphisms, such that $(- \act M): \cC \to \cM$ is exact for each $M \in \cM$. A similar notion holds for right module categories.
We also assume that module functors between such module categories are additive in each slot.

\begin{remark}\label{rem:act-faithful}
In the finite tensor case above,  $X\act M$ is nonzero in $\cM$, for any nonzero $X \in \cC$, and any nonzero $M\in\cM$.
Indeed, by way of contradiction, if $X\act M=0$, then $(X^* \otimes X)\act M=0$. Since $\text{ev}^L_X:X^*\otimes X\rightarrow\one$ is epic, and the functor $(-\act M): \cC \to \cM$ is exact, we get that $0=(X^* \otimes X)\act M \rightarrow \one\act M \cong M$ is epic. This implies that  $M=0$, a contradiction. 
\end{remark}

\subsubsection{Exact module categories} A finite module category $(\cM, \act)$ over a finite tensor category $\cC$ is called {\it exact} if for any projective object $P \in \cC$ and any object $M \in \cM$, we have that the object $P \act M$ is projective in $\cM$.

\smallskip

We say that an algebra $A$ in $\cC$ is {\it exact} if the left $\cC$-module category $\mathsf{Mod}$-$A(\cC)$ is exact.

\subsubsection{Categories of exact module functors} \label{ss:Fun-C}
Let $\mathsf{Fun}_{\cC|}(\cM,\cN)$ denote the category of right exact \linebreak $\cC$-module functors from a left $\cC$-module category $\cM$ to a left $\cC$-module category $\cN$. Here, morphisms are natural transformations between such functors, compatible with left $\cC$-actions. 

\smallskip

If $\cC$ is a finite tensor category, and $\cM$ is an indecomposable, exact left $\cC$-module category, then $\mathsf{Fun}_{\cC|}(\cM,\cM)$ is a finite tensor category \cite{etingof2004finite}. Here,  $(F,s) \otimes^{\mathsf{Fun}_{\cC|}(\cM,\cM)} (F',s') := (FF', s^{FF'})$, where
\[
s^{FF'}_{X,M} := s_{X,F'(M)} \circ F(s'_{X,M}): FF'(X \triangleright M) \to X \triangleright FF'(M),
\]
for $X \in \cC$ and $M \in \cM$. We have also that $\one^{\mathsf{Fun}_{\cC|}(\cM,\cM)}:=\bigl(\text{Id}_{\cM},\{\id_{X \triangleright M}\}_{X,M}\bigr)$.

\smallskip

In particular, we obtain that: 
\begin{equation} \label{eq:FunCop}
\mathsf{Fun}_{\cC|}(\cC_{\text{reg}},\cC_{\text{reg}}) \; \overset{\otimes}{\simeq} \; \cC^{\otimes \text{op}},  
\end{equation}
where $(F,s): \cC_{\text{reg}} \to \cC_{\text{reg}}$ is sent to $F(\one)$, and  $Z$ is sent to $((- \otimes Z), s:=\id): \cC_{\text{reg}} \to \cC_{\text{reg}}$.

\subsubsection{Closed module categories and internal Homs}  A left $\cC$-module category $(\cM, \triangleright)$ is  {\it closed} if, for each $M \in \cM$, the functor   $(- \triangleright M): \cC \to \cM$ has a right adjoint, denoted by 
\[
\underline{\textnormal{Hom}}(M,-): \cM \to \cC.
\]
That is, there exists a bijection
\begin{equation} \label{eq:internalHom}
\zeta := \zeta_{X,N}:= \Hom_{\cM}(X \triangleright M, N) \; \overset{\sim}{\to} \; \Hom_{\cC}(X,\;  \underline{\textnormal{Hom}}(M, N)),
\end{equation}
natural in   $X \in \cC$ and in $N \in \cM$.
The object 
 $\underline{\textnormal{Hom}}(M, N) \in \cC$  is  the {\it internal Hom} of $M$ and $N$.

 \medskip
 
We label the following morphisms:
 \[
 \begin{array}{rl}
\coev_{\one,M}\;  &:= \zeta(\id_M): \; \one \to \underline{\textnormal{Hom}}(M, M) \; \in \cC,\\[.4pc]
\ev_{M,N} \; &:= \zeta^{-1}(\id_{\underline{\textnormal{Hom}}(M,N)}): \;\underline{\textnormal{Hom}}(M,N) \act M \to N \; \in \cM,\\[.4pc]
\ev_{M,N,P} \; &:= \ev_{N,P} \,(\id_{\underline{\textnormal{Hom}}(N,P)} \act \ev_{M,N}): \;\underline{\textnormal{Hom}}(N,P) \otimes \underline{\textnormal{Hom}}(M,N) \act M \to P \; \in \cM,\\[.4pc]
\text{comp}_{M,N,P} \; & := \zeta(\ev_{M,N,P}): \; \underline{\textnormal{Hom}}(N,P) \otimes \underline{\textnormal{Hom}}(M,N)  \to \underline{\textnormal{Hom}}(M,P) \; \in \cC.
 \end{array}
 \]
 One useful identity is that for all $M \in \cM$:
 \begin{equation} \label{eq:evMMcoev1M}
\ev_{M,M} \; (\coev_{\one, M} \act \id_M) = \id_M,
 \end{equation}
by the naturality of $\zeta_{X,N}$ in $X$. Another useful identity is that for all $M,N,P \in \cM$:
 \begin{equation} \label{eq:compMNPiden}
\ev_{M,P} \; (\text{comp}_{M,N,P} \act \id_M) = \ev_{N,P} \; (\id_{\underline{\text{Hom}}(N,P)} \act \ev_{M,N}).
 \end{equation}
by the naturality of $\zeta_{X,N}$ in $X$.
If $\cC$ is rigid, we have also that, for each $X \in \cC$, $M,N \in \cM$:
 \begin{equation} \label{eq:internalHom-2}
X \otimes \; \underline{\textnormal{Hom}}(M, N) \; \cong \; \underline{\textnormal{Hom}}(M, X  \triangleright N).
\end{equation}

Note that finite left module categories over finite tensor categories $\cC$ are always closed. For example, if $\cM = \cC_{\text{reg}}$ with $Z \in \cC_{\text{reg}}$, then $\underline{\textnormal{Hom}}(Z,Z) = Z \otimes Z^*$.


\subsection{Braided module categories}  \label{sec:brmod}
See \cite{laugwitz2023reflective} for more details here.
Take $\cC:=(\cC, \otimes, \one, c)$ to be a braided finite tensor category.  We say that a left $\cC$-module category $(\cM, \act)$ is {\it braided} if it is equipped with a natural isomorphism, 
$$e:=\{e_{X,M}\colon  X \act M \overset{\sim}{\to}  X \act M\}_{X \in \cC, M \in \cM} \; \; \text{({\it braiding})},$$ 
 such that the following axioms hold, for each $X,Y \in \cC$ and $M \in \cM$:
\begin{align} \label{eq:brmod1}
e_{X \otimes Y, M}
  &=  (\id_X \act e_{Y,M}) (c_{Y,X} \act \id_M)  (\id_Y \act e_{X,M})(c_{Y,X}^{-1} \act \id_M),\\[.2pc]
\label{eq:brmod2}
 e_{X,Y \act M}
  &= (c_{Y,X} \act \id_M)  (\id_Y \act e_{X,M})  (c_{X,Y} \act \id_M).
 \end{align}

Using \eqref{eq:braid-unit}, we obtain that for each $M \in \cM$:
\begin{align}
e_{\one, M}  = \id_M. 
\label{eq:braidmod-unit}
\end{align}

\begin{example} \label{ex:M=C,double}
Suppose that $(\cC,\otimes,\one,c)$ and $(\cC',\otimes',\one',c')$ are braided finite tensor categories and \linebreak $F: \cC \to \cC'$ is a braided tensor functor. Then, $\cC'$ is a braided left $\cC$-module category with $X \act M := F(X) \otimes' M,$ 
and it is braided with
$e_{X,M} := c'_{M,F(X)} \circ c'_{F(X),M}$,
for all $X \in \cC$ and $M \in \cC'$.

Taking $F$ to be identity functor, for instance, implies that the regular left $\cC$-module category $\cC_{\reg}$ is braided with $e = c^2$.
\end{example}

A {\it braided $\cC$-module functor} between braided left $\cC$-module categories $(\cM,\act, e)$ and $(\cM',\act', e')$ is a left $\cC$-module functor $(F, s)\colon  (\cM, \act) \to (\cM', \act')$ such that, for all $X \in \cC$ and $M \in \cM$:
\begin{equation} \label{eq:braidmodfunc}
e'_{X,F(M)} \circ s_{X,M} = s_{X,M} \circ  F(e_{X,M}).
\end{equation}
An {\it equivalence (resp., isomorphism) of  braided $\cC$-module categories} is a braided $\cC$-module functor that yields an equivalence (resp., isomorphism) of the underlying categories.


\subsection{Reflective centers} \label{sec:refcenter} 
See \cite{laugwitz2023reflective} for more details here.
Take $\cC:=(\cC, \otimes, \one, c)$ to be a braided finite tensor category. An important example of a braided module category over $\cC$ is the {\it reflective center}  $\cE_{\cC}(\cM)$ of a left $\cC$-module category $(\cM, \act)$. Its objects are pairs $(M, e^M)$, where $M \in \cM$ and $e^M:=\{e^M_X \colon  X \act M \overset{\sim}{\to} X \act M\}_{X \in \cC}$ is a natural isomorphism, called a  {\it reflection}, satisfying 
\begin{equation} \label{eq:reflection}
\begin{array}{c}
e^M_{X \otimes Y}  \;=  \; (\id_X \act e^M_Y) (c_{Y,X} \act \id_M) (\id_Y \act e^M_X)(c_{Y,X}^{-1} \act \id_M).
 \end{array}
\end{equation}
Morphisms $(M, e^M) \to (N, e^N)$ of $\cE_{\cC}(\cM)$ are given by
  $f \in \Hom_{\cM}(M,N)$ such that, for all $X \in \cC$:
 \begin{equation} \label{eq:ECMmorp}
 (\id_X \act f) \; e^M_X = e^N_X \; (\id_X \act f).
 \end{equation}
 The left $\cC$-action on $\cE_{\cC}(\cM)$ is 
$Y \act (M,e^M):=(Y \act M, e^{Y \act M})$,
for  $Y \in \cC$, $(M,e^M) \in \cE_{\cC}(\cM)$, where
\begin{equation} \label{eq:ECMact}
e^{Y \act M}_X \; := \; (c_{Y,X} \act \id_M)  (\id_Y \act e_X^M)  (c_{X,Y} \act \id_M),
\end{equation}
for each $X \in \cC$. The braiding of $\cE_{\cC}(\cM)$ is 
$e^{\cE_{\cC}(\cM)}_{Y,(M,e^M)} := e_Y^M:  (Y \act M,e^{Y \act M}) \to 
(Y \act M,e^{Y \act M})$,
for each $Y \in \cC$ and $(M,e^M) \in \cE_{\cC}(\cM)$.

\begin{remark} \label{rem:Ec(C)Z(C)}
The reflective center $\cE_\cC({}_{\cC}(\cC_{\textnormal{reg}}))$ of the regular left $\cC$-module category is isomorphic to  $\cZ(\cC)$, as categories. This is given by sending $(V,e^V)$ to $(V,c^V)$, where $c^V_X:= c_{V,X}^{-1} \circ e^V_X$ for $X \in \cC$.
\end{remark}


\section{Category equivalences for module categories} 

In this section, we provide background material and preliminary results on algebras formed by end objects [$\S$\ref{ss:Endmod}] and on Eilenberg-Moore (EM-)categories attached to module categories [$\S$\ref{ss:EMcat}]. Next, we establish a monadicity result for such EM-categories [$\S$\ref{ss:monadicity}]. This monadicity is then used to prove a category equivalence between a certain Deligne product of categories and an EM-category, both attached to a module category [$\S$\ref{ss:con-DelEM}]. Lastly, we prove a category equivalence between a reflective center and an EM-category, both attached to a module category [$\S$\ref{ss:con-refEM}].


\subsection{End algebras} \label{ss:Endmod}

Recall the (co)end objects from $\S$\ref{ss:coendBTC}, and consider the end objects in $\cC$ attached to $\cM$ below:
\[ 
\mathtt{E}_{\cM}:= \textstyle \int_{M \in \cM} \underline{\textnormal{Hom}}(M,M).
\]
By their universal constructions, $\mathtt{E}_{\cM}$ is attached to canonical morphisms in~$\cC$: 
\[ 
\{\pi_{\cM}(M): \mathtt{E}_{\cM} \to \underline{\textnormal{Hom}}(M,M)\}_{M \in \cM}.
\]

\smallskip

Using \eqref{eq:EcCc}, we then have for the regular left $\cC$-module category that: 
\begin{equation} \label{eq:EcCreg}
\; \; {\tt E}_{\cC_{\textnormal{reg}}}^* \;  =  \; {\tt E}_\cC^*  \; =  \;  \textstyle \bigl(\int_{Z \in \cC}  \underline{\textnormal{Hom}}(Z,Z) \bigr)^* \;   \cong \; \textstyle \bigl(\int_{Z \in \cC}  Z \otimes Z^* \bigr)^*  \; \cong \; {\tt C}_\cC.
\end{equation}

\smallskip

The following result is straightforward to verify; see, e.g., \cite{shimizu2020further}.

\begin{lemma} \label{lem:EmAlg}
The  end object $\mathtt{E}_{\cM}$ exists, and is an algebra in $\cC$. For $M \in \cM$, the algebra operations of $\mathtt{E}_{\cM}$ are given by:

\vspace{-.2in}

\[
\begin{array}{rl}
\pi_{\cM}(M) \circ m_{\mathtt{E}_{\cM}} & = \; \textnormal{comp}_{M,M,M}\; (\pi_{\cM}(M) \otimes \pi_{\cM}(M)),\\[.4pc]
\pi_{\cM}(M) \circ u_{\mathtt{E}_{\cM}} & = \; \coev_{\one, M}.
\end{array}
\]

\vspace{-.2in}

\qed
\end{lemma}


\subsection{Eilenberg-Moore (EM-)categories} 
\label{ss:EMcat} Let $(\cM, \act)$ be a left $\cC$-module category, and take an algebra $(A,m_A,u_A)$ in  $\cC$. Then, $(A \act -): \cM \to \cM$ forms a monad on $\cM$ with multiplication and unit natural transformations given by:
\[
\hspace{-.07in}
\begin{array}{rl}
\mu^{A \act -}: (A \act -) \circ (A \act -) \Rightarrow (A \act -),
&\{\mu^{A \act -}_M:= m_A \act \id_M: A \otimes A \act M \to A \act M\}_{M \in \cM},\\[.6pc]
\eta^{A \act -}: \text{Id}_{\cM} \Rightarrow (A \act -),
&
\{\eta^{A \act -}_M:= u_A \act \id_M : M \to A \act M\}_{M \in \cM}.
\end{array}
\]

Let $A$-$\mathsf{Mod}(\cM)$ denote the corresponding {\it Eilenberg-Moore (EM) category}. 
Namely, objects are pairs $(M, \xi_M)$, where $M \in \cM$, and  $\xi_M:= \xi^A_M: A \act M \to M$ is a morphism in $\cM$, such that
\begin{equation} \label{eq:EMobject}
\xi_M \circ \mu^{A \act -}_M \; = \; \xi_M \circ (\id_A \act \xi_M), \quad \quad \quad \quad \xi_M \circ \eta^{A \act -}_M \; = \;  \id_M.
\end{equation}
A morphism from $(M, \xi_M)$ to $(N, \xi_{N})$ of $A$-$\mathsf{Mod}(\cM)$ is a morphism $f: M \to N$ in~$\cM$, such that 
\begin{equation} \label{eq:EMmorphism}
f \; \xi_M \; = \; \xi_{N}  (\id_A \act f).
\end{equation}

We also have another useful result due to Davydov--Nikshych.

\begin{lemma} \label{lem:DN} 
\cite[Lemma~3.2]{davydov2013picard} 
For an algebra $A$ in $\cC$, and a left $\cC$-module category $\cM$, we have that the following is an equivalence of categories:
\begin{align*}
\mathsf{Fun}_{\cC|}(\mathsf{Mod}\text{-}A(\cC), \cM) &\overset{\sim}{\to} A\text{-}\mathsf{Mod}(\cM),\\[-0.5pc]
F &\mapsto \bigl(F(A), \; 
\xymatrix{
\xi_{F(A)}^A: A \act F(A) \cong F(A \otimes A) \ar[r]^(.8){F(m_A)} & A\bigr).}
\end{align*}

\vspace{-.3in}

\qed
\end{lemma}

Now, we consider actions $\act$ relative to $A$.
For $(X, \vartriangleleft^A_X) \in \mathsf{Mod}\text{-}A(\cC)$ and $(M, \xi^A_M) \in A\text{-}\mathsf{Mod}(\cM)$, consider the following coequalizer in $\cM$:
\begin{equation}\label{eq:coeq-over-A}
X \act_A M 
:= \textnormal{coeq}\left(X \otimes A \act M 
\doublerightarrow{\vartriangleleft^A_X \; \act \; \id}{\id \; \otimes \; \xi^A_M}
X \act M\right).  
\end{equation}

The proof of the following result is standard.

\begin{lemma} \label{lem:coeq-id}
For an algebra $A$ in $\cC$, and $(M, \xi^A_M) \in A\text{-}\mathsf{Mod}(\cM)$, we have that  $A \act_A M \cong M$. \qed
\end{lemma}


Let $\chi:A\act M\rightarrow A \act_A M$ denote the canonical projection map of the coequalizer above, and call the isomorphism $\alpha: A\act_A M\rightarrow M$. As a consequence of the lemma above, we have that 
    $\alpha\circ \chi = \xi^A_M$,
and the standard result below.

\begin{lemma}\label{lem:for-later-prop}
For an algebra $A$ in $\cC$, and $(M, \xi^A_M) \in A\text{-}\mathsf{Mod}(\cM)$, we have that 
\begin{equation}\label{eq:A-action}
  \xi^A_M \circ (\id_A \act \alpha) = \alpha \circ (m_A \act_A \id_M) :A \act (A\act_A M) \rightarrow M. 
\end{equation}

\vspace{-.25in}

\qed
\end{lemma}


By precomposing \eqref{eq:A-action} with $\id_A\act\alpha^{-1}:A\act M \rightarrow A\act(A\act_A M)$, we get the following consequence. For an algebra $A$ in $\cC$, and $(M, \xi^A_M) \in A\text{-}\mathsf{Mod}(\cM)$, we have that 
\begin{equation}\label{eq:xi-equals-action}
    \xi_M^A = \alpha\circ (m_A \act_A \id_M)\circ (\id_A\act \alpha^{-1}) : A\act M \rightarrow M.
\end{equation}

Here is a useful lemma, which will be employed later; see, e.g., \cite[$\S$3.3]{shimizu2019non}.
\begin{lemma}\label{lem:AlgResEquiv}
Take an algebra morphism  $ \phi: A \to A'$ in  $\cC$, and let $\cM$ be a left $\cC$-module category. Consider the functor:
\[
\textnormal{Res}_\phi: A'\text{-}\mathsf{Mod}(\cM) \to A\text{-}\mathsf{Mod}(\cM), \quad \quad
(M', \; \xi_{M'}) \mapsto (M', \; \xi_{M'} (\phi \act \id_{M'})).
\]
Then, $\phi$ is an isomorphism of algebras in $\cC$ if and only if $\textnormal{Res}_\phi$ is an equivalence of categories. 
\end{lemma}
\begin{proof}
($\Rightarrow$) As $\phi$ is an algebra map, so is its inverse $\phi^{-1}$. Now a direct check shows that the corresponding restriction functor $\text{Res}_{\phi^{-1}}$ is the inverse of $\text{Res}_{\phi}$.

($\Leftarrow$) Let $A'_{\phi}$ denote $A'$ considered as an $A'\text{-}A$-bimodule using the map $\phi$. Using \eqref{eq:coeq-over-A}, the functor 
\[ A'_{\phi}\act_A - : A\text{-}\Mod(\cM) \rightarrow A'\text{-}\Mod(\cM) \]
is defined. One can check that $A'_{\phi}\act_A -$ is the left adjoint of $\text{Res}_{\phi}$. Let $\eta$ denote the unit of this adjunction. For any $M\in \cM$, the component of $\eta$ at $A\act M \in A\text{-}\Mod(\cM)$, namely,
\[ \eta_{A\act M} : A\act M \xrightarrow{\eta_{A\act M}} 
\text{Res}_{\phi}(A'_{\phi} \act_A (A\act M)) 
\overset{\text{Lem}.~\textnormal{\ref{lem:coeq-id}}}{\cong} 
\text{Res}_{\phi}(A' \act M) \cong \text{Res}_{\phi}(A')\act M \]
is equal to $\phi\act\id_M$. As $\text{Res}_{\phi}$ is an equivalence, $\eta_{A\act M}$ is an isomorphism. 
Thus, $\phi \act M$, which is equal to $\eta_{A\act M}$, is an isomorphism. On the other hand, $(-\act M)$ is exact by definition, and is faithful by Remark~\ref{rem:act-faithful}. Since exact, faithful functors between finite abelian categories reflect isomorphisms (see, e.g., \cite[Theorem~7.1]{mitchell1965categories}), we obtain that $\phi$ is an isomorphism, as desired.
\end{proof}


\subsection{Monadicity result} \label{ss:monadicity}
Let $\cC$ be a finite tensor category, and $\cM$ a left $\cC$-module category. By \cite[Theorem~3.4]{shimizu2020further} we have functors $\rho:\cC \rightarrow \mathrm{Fun}(\cM,\cM) $ and its right adjoint $\rho^{\mathrm{ra}}$ given by 
 \[ \rho(X) = (X \act -),\qquad \rho^{\mathrm{ra}}(F) = \textstyle \int_{M\in\cM} \iHom(M,F(M)). \]
Now, we obtain the following monadicity result. 

\begin{proposition} \label{prop:monad}
Let $\cC$ be a finite tensor category and $\cM$ an indecomposable, exact left  $\cC$-module category. Then, the following statements hold.
\begin{enumerate}[\upshape (a)]
\item The adjunction $\rho\dashv\rho^{\ra}$ is monadic.

\smallskip

\item The EM-category on the monad $\rho^{\textnormal{ra}} \rho$ on $\cC$ is equivalent to $\mathsf{Mod}\text{-}\mathtt{E}_{\cM}(\cC)$, and the following comparison functor is an equivalence:
\[
\kappa: \mathsf{Fun}(\cM,\cM) \to \mathsf{Mod}\text{-}\mathtt{E}_{\cM}(\cC), \qquad F \mapsto \textstyle \int_{Q \in \cM} \underline{\textnormal{Hom}}_{\cM}(Q, F(Q)),
\]

\item The quasi-inverse of $\kappa$ is
\[
\Lambda:  \mathsf{Mod}\text{-}\mathtt{E}_{\cM}(\cC) \to \mathsf{Fun}(\cM,\cM), \qquad (X, \vartriangleleft_X^{\mathtt{E}_{\cM}}) \mapsto (X \act_{\hspace{0.03in} \mathtt{E}_{\cM}} -).
\]
Here, the action $\mathtt{E}_{\cM}$ on $M \in \cM$ is given by 
\[
\xymatrix@C=4pc{
\xi_M^{\mathtt{E}_{\cM}}:
\mathtt{E}_{\cM} \act M
\ar[r]^(.47){\pi_{\cM}(M) \; \act \; \id}
&
\underline{\textnormal{Hom}}(M,M) \act M
\ar[r]^(.65){\textnormal{ev}_{M,M}}
& M. 
}
\]
\end{enumerate}
\end{proposition}

\begin{proof}
(a) As $\cM$ is exact and indecomposable, $\rho^{\ra}$ is exact and faithful  by \cite[Theorem~3.4]{shimizu2020further}. As the underlying categories $\cC$ and $\mathrm{Fun}(\cM,\cM)$ are abelian and $\rho^{\ra}$ is additive, faithful and exact, Beck's monadicity theorem \cite[\S VI.7]{mac1998categories} implies that the adjunction $\rho\dashv\rho^{\ra}$ is monadic.

For the remaining proof, we basically mimic the argument provided in \cite[Proposition~6.1]{bruguieres2011exact}. The main difference is that here we apply the argument to the functor $\rho$ whose codomain may not be rigid while the result in \cite{bruguieres2011exact} is for rigid categories. To proceed, note that since $\rho$ is strong monoidal, $\rho^{\ra}$ is a lax monoidal functor. Thus, $\rho^{\ra}\dashv\rho$ is a so called monoidal adjunction. Consequently, $(\rho^{\ra})^{\oop} \dashv \rho^{\oop}$ is a comonoidal adjunction \cite[\S2.5]{bruguieres2011hopf}. Furthermore, the category $\cC$ is rigid and for any $X\in\cC$, the object $\rho(X)\in\mathrm{Fun}(\cM,\cM)$ is rigid with duals given by $\rho(X^*)$ and $\rho({}^*X)$. Thus, the conditions of \cite[Theorem~3.14]{bruguieres2007hopf} are satisfied. Consequently, the bimonad $(\rho^{\ra})^{\oop} \circ \rho^{\oop}$ is a Hopf monad on $\cC$, or equivalently, the adjunction $(\rho^{\ra})^{\oop} \dashv \rho^{\oop}$ is Hopf monoidal.

\smallskip

(b) With the above, we have  that $\rho^{\ra}(\id_{\cM})= \mathtt{E}_{\cM}$ is a commutative algebra in $\cZ(\cC)$ (with some half-braiding $\sigma$) by \cite[Theorem~6.6]{bruguieres2011hopf}. Using $\sigma$, the category $\Mod$-$\mathtt{E}_{\cM}(\cC)$ can be endowed with a tensor product. Also, \cite[Theorem~6.6]{bruguieres2011hopf} implies that $\mathrm{Fun}(\cM,\cM)$ and $\Mod$-$\mathtt{E}_{\cM}(\cC)$ are monoidally equivalent via the comparison functor $\kappa$ that is given by 
\[ \kappa(F) = \left( \rho^{\ra}(F), \; \rho^{\ra}_{F,\id_{\cM}}: \rho^{\ra}(F) \otimes \mathtt{E}_{\cM} \rightarrow \rho^{\ra}(F) \right) \]
where $\rho^{\ra}_{\text{-},\text{-}}$ is the monoidal structure of $\rho^{\ra}$. Lastly, $\rho^{\ra}\rho$ is monoidally isomorphic to the free module functor $\cC \rightarrow \Mod\text{-}\mathtt{E}_{\cM}(\cC)$, where $X$ is sent to $X\otimes \mathtt{E}_{\cM}$. Thus, the EM-category on $\rho^{\ra}\rho$ is equivalent to $\mathsf{Mod}$-$\mathtt{E}_{\cM}(\cC)$, as desired.

\smallskip

(c) By the proof of \cite[\S VI.7, Theorem~1]{mac1998categories}, the value of the quasi-inverse $\Lambda$ of $\kappa$ at  $(X, \vartriangleleft_X)$ in $\mathsf{Mod}$-$\mathtt{E}_{\cM}$ is given by the coequalizer of the following maps:
\[ 
\xymatrix{
\varepsilon_{\rho(X)},\; \rho(\vartriangleleft_X^{\mathtt{E}_{\cM}}): 
\rho \rho^{\ra} \rho (X) \cong \rho(X \otimes \mathtt{E}_{\cM}) \ar[r]
& \rho(X).
}
\]
Here, $\varepsilon$ is counit of the adjunction $\rho\dashv\rho^{\ra}$. Using the description of $\varepsilon$ provided in \cite[\S3.3]{shimizu2020further}, we get that $\varepsilon_{\rho(X)}(M)$ is given by
{\small
\[
\xymatrix@R=1pc@C=5pc{
\varepsilon_{\rho(X)}(M):  \textstyle  \rho \rho^{\ra} \rho (X)(M) = \left(\int_{N\in\cM} \iHom(N,X\act N)\right) \act M  \ar[r]^(.66){\pi_\cM(M) \; \act\;  \id}
& \iHom(M,X\act M)\act M  \ar[d]^{\ev_{X\act M,M}}\\
 & X\act M =\rho(X)(M).
 }
\]
}

\vspace{-.1in}

\noindent Using the isomorphism 
\[ 
\textstyle X\otimes \mathtt{E}_{\cM} 
\; = \; X\otimes \int_{N\in\cM} \iHom(N,N) \; \cong \; \int_{N\in\cN} X\otimes \iHom(N,N) \; \overset{\eqref{eq:internalHom-2}}{\cong} \; \int_{N\in\cM} \iHom(N,X\act N),
\]
and naturality, the map $\varepsilon_{\rho(X)}(M): \rho \rho^{\ra} \rho (X)(M) \rightarrow \rho(X)(M)$ can be identified with:
\[ 
\xymatrix@C=7pc{
f_1: X \act (\mathtt{E}_{\cM} \act M) \ar[r]^(.48){\id \; \act\;  \pi_\cM(M)\;  \act \; \id}
& X\act(\iHom(M,M)\act M) \ar[r]^(.6){\id  \; \act\;  \ev_{M,M}} 
& X\act M.
}
\]
On the other hand, the map $\rho(\vartriangleleft_X^{\mathtt{E}_{\cM}})$ can be identified with:

\vspace{-.2in}

\[ 
\xymatrix@C=5pc{
f_2: X \act (\mathtt{E}_{\cM} \act M) \ar[r]^{\sim} 
& (X \otimes \mathtt{E}_{\cM})\act M \ar[r]^(.57){\vartriangleleft_X^{\mathtt{E}_{\cM}} \; \act \; \id} 
& X\act M.
}
\]
Now comparing with \eqref{eq:coeq-over-A}, it follows that $\Lambda(X)(M)=\text{coeq}(f_1,f_2)$ equals $X\act_{\mathtt{E}_{\cM}} M$.
\end{proof}


\subsection{Connection between Deligne products and EM-categories} \label{ss:con-DelEM}
Recall the end objects $\mathtt{E}_\cM$ from $\S$\ref{ss:Endmod}. Our goal here is to express the EM-category $\mathtt{E}_\cM$-$\mathsf{Mod}(\cM)$ as a Deligne product of categories attached to $\cM$. To proceed, note that $\Fun(\cM,\cM)$ is a left $\cC$-module category via 
\[
(X\act T)(M):=X\act T(M)\]
for $X\in\cC,M\in\cM$ and $T\in\Fun(\cM,\cM)$.
We will need the following preliminary result.

\begin{lemma}\label{lem:key-prop}
We have the following equivalence of categories
\begin{equation*}
\Phi: \cM  \boxtimes  \mathsf{Fun}_{\cC|}(\cM, \cM) 
\xrightarrow{\sim}
\mathsf{Fun}_{\cC|}(\mathsf{Fun}(\cM, \cM), \cM), \quad M\boxtimes F \mapsto [T \mapsto F(T(M))],
\end{equation*}
Its quasi-inverse is given by
\begin{equation*}
\widehat{\Phi}:  
\mathsf{Fun}_{\cC|}(\mathsf{Fun}(\cM, \cM), \cM)
\rightarrow \cM  \boxtimes  \mathsf{Fun}_{\cC|}(\cM, \cM), \quad 
G \mapsto \textstyle \int_{M\in\cM} M \boxtimes F_{G,M}(?), 
\end{equation*}
for $F_{G,M}(N) := G( \Hom_{\cM}(-,M)^* \otimes_{\kk}  N) \in \cM$, where  $N \in \cM$.
\end{lemma}

\begin{proof}
To start, note that $\Phi(M\boxtimes F)$ is a left $\cC$-module functor as, for $T\in\mathsf{Fun} (\cM,\cM)$, we get:
\begin{equation} \label{eq:sec4prelim}
\begin{array}{rll}
X\act \left[\Phi(M\boxtimes F)(T)\right] & = X\act F(T(M))& \cong F[X\act T(M)]\\[.2pc]
& = F\left[ (X\act T)(M) \right] &= [\Phi(M\boxtimes F)](X\act T).
\end{array}
\end{equation}
Thus, $\Phi$ is well defined.

To show that $\widehat{\Phi}$ is well-defined, take 
$G\in\ \mathsf{Fun}_{\cC|}(\mathsf{Fun}(\cM, \cM), \cM)$ and $M \in \cM$. Then, the functor $F_{G,M}(?): \cM \to \cM$ is a left $\cC$-module functor because
\[
\begin{array}{rll}
F_{G,M}(X\act N) & = G(\Hom_{\cM}(-,M)^*\otimes_{\kk} (X\act N)) 
 &\stackrel{(\dagger)}{\cong} G(X\act (\Hom_{\cM}(-,M)^*\otimes_{\kk} N)) \\[.2pc]
& \cong X \act G(\Hom_{\cM}(-,M)^*\otimes_{\kk} N) &= X\act F_{G,M}(N). 
\end{array}
\]
The isomorphism $(\dagger)$ holds because the $\cC$-action and $\FdVec$-action on $\cM$ commute. The last isomorphism holds because $G$ is a left $\cC$-module functor. Thus, $\widehat{\Phi}$ is well defined.

Next we show that $\widehat{\Phi}$ is the quasi-inverse of $\Phi$:
\begin{align*}
  (\widehat{\Phi}\circ\Phi)(M\boxtimes F(?)) 
  & \; = \; \widehat{\Phi}\left([T\mapsto F(T(M))]\right)
 \;  =  \; \textstyle  \int_{P\in\cM} P \boxtimes F\left( \Hom_{\cM}(P,M)^* \otimes_{\kk} \; ? \hspace{0.01in}\right) \\
  & \; \stackrel{(\dagger)}{\cong} \;  \textstyle  \int_{P\in\cM} P \boxtimes \Hom_{\cM}(P,M)^* \otimes_{\kk} F(?) 
 \; \stackrel{\eqref{eq:coend-1}}{\cong} \; M \boxtimes F(?).
\end{align*}
Here, the isomorphism $(\dagger)$ holds because $F$ is $\kk$-linear.
On the other hand,
\begin{align*}
  (\Phi\circ \widehat{\Phi})(G) 
  &= \textstyle \Phi\left(\int_{M\in\cM} M \boxtimes G( \Hom_{\cM}(-,M)^* \otimes_{\kk} \; ? \hspace{0.01in})\right)\\
  & = \textstyle  [T \mapsto \int_{M\in\cM} G( \Hom_{\cM}(-,M)^* \otimes_{\kk} T(M) ) ]
  \; \stackrel{\eqref{eq:coend-2}}{\cong}\;  [T \mapsto G(T(-))] = G.
\end{align*}
Hence, the proof is finished.
\end{proof}

\begin{remark}
The above result can also be obtained using the {\it module ends} introduced by Bortolussi-Mombelli \cite{bortolussi2021coend}. 
\end{remark}

Our main result here is the following.

\begin{proposition} \label{prop:HM} 
We have an equivalence of categories:
\[
H_{\cM}: \cM \boxtimes  \mathsf{Fun}_{\cC|}(\cM, \cM) \longrightarrow {\tt E}_{\cM}\text{-}\mathsf{Mod}(\cM), \qquad
M \boxtimes (F,s) \mapsto (F(M), \; \xi^{{\tt E}_{\cM}}_{F(M)}).
\]
Here, $\xi^{{\tt E}_{\cM}}_{F(M)}$ is defined as follows:

\vspace{-.2in}

{\small
\[
\xymatrix@C=7pc@R=1.3pc{
{\tt E}_{\cM} \act F(M)
\ar@{-->}[r]^{\xi^{{\tt E}_{\cM}}_{F(M)}}
\ar[d]_{s^{-1}_{\mathtt{E}_{\cM},M}}
&
F(M)
\\
F(\mathtt{E}_{\cM}\act M) 
\ar[r]^(.42){F(\pi_{\cM}(M) \; \act \; \id)}
&
 F\bigl(\underline{\textnormal{Hom}}(M,M)\act M\bigr) \ar[u]_{F(\ev_{M,M})}.
}
\]
}

\medskip

\noindent For $[f:M \to M'] \boxtimes [\phi:(F,s) \Rightarrow (F',s')] \in \cM  \boxtimes   \mathsf{Fun}_{\cC|}(\cM, \cM)$, we define:


{\small
\[
\xymatrix@C=6pc{
H_{\cM}(f \boxtimes \phi) \; : \bigl(F(M), \; \xi^{{\tt E}_{\cM}}_{F(M)}\bigr)
\ar[r]^(.57){F'(f) \; \circ  \; \phi_M}
& 
\bigl(F'(M'), \; \xi^{{\tt E}_{\cM}}_{F'(M')}\bigr).
}
\]
}
\end{proposition}

\begin{proof}
We will show that $H_{\cM}$ is the composition of equivalences below, thus is an equivalence itself. 

\vspace{-.15in}

{\small
\[
\xymatrix@C=1.7pc@R=1pc{
\cM  \boxtimes  \mathsf{Fun}_{\cC|}(\cM, \cM)
\ar[r]^(.46){\textnormal{(i)}}_(.46){\sim}
&
\mathsf{Fun}_{\cC|}(\mathsf{Fun}(\cM, \cM), \cM)
\ar[r]^(.50){\textnormal{(ii)}}_(.50){\sim}
&
\mathsf{Fun}_{\cC|}(\mathsf{Mod}\text{-}{\mathtt E}_{\cM}(\cC), \cM)
\ar[r]^(.60){\textnormal{(iii)}}_(.60){\sim}
&
{\mathtt E}_{\cM}\text{-}\mathsf{Mod}(\cM).
}
\]
}

\noindent The equivalence (i) is the map $\Phi$ from Lemma~\ref{lem:key-prop}, (ii) is given by pre-composition with the equivalence $\Lambda: \mathsf{Mod}\text{-}\mathtt{E}_{\cM}(\cC) \overset{\sim}{\to} \mathsf{Fun}(\cM, \cM)$ from Proposition~\ref{prop:monad}(c),
and (iii) is defined by evaluation at $\mathtt{E}_{\cM}$ according to Lemma~\ref{lem:DN}. Combining these yields:

\vspace{-.1in}

{\small
\begin{align*}
    M \boxtimes F 
    & \xmapsto{\text{(i)}}  \left[ T \mapsto F(T(M)) \right] 
    \xmapsto{\text{(ii)}} \left[ X  
    \mapsto F(X \act_{\hspace{0.03in}{\mathtt E}_{\cM}} M) \right] 
    \xmapsto{\text{(iii)}} F({\mathtt E}_{\cM} \act_{\hspace{0.03in}{\mathtt E}_{\cM}} M ) \; \stackrel{\text{Lem.}~\textnormal{\ref{lem:coeq-id}}}{\cong} \; F(M).
\end{align*}
}

\smallskip
 
To finish the proof, we also need to show that the ${\mathtt E}_{\cM}$-action on $F(M)$ obtained using the maps above is same as $\xi_{F(M)}^{{\mathtt E}_{\cM}}$. 
First, after applying the equivalence (i), the $\cC$-module structure of the resulting functor is described in Lemma~\ref{lem:key-prop}; see \eqref{eq:sec4prelim}. Next, after applying (ii), the $\cC$-module structure of the resulting functor $[ X \mapsto F(X \act_{\hspace{0.03in}{\mathtt E}_{\cM}} M) ] $ is given using (i). Explicitly, it is:
\[ 
Y\act F(X \act_{\hspace{0.03in}{\mathtt E}_{\cM}} M) \xrightarrow{s^{-1}} F( Y \act ( X \act_{\hspace{0.03in}{\mathtt E}_{\cM}} M)) \cong F((Y\otimes X) \act_{\hspace{0.03in}{\mathtt E}_{\cM}} M ). 
\]
Now, by Lemma~\ref{lem:DN}, the $\mathtt{E}_{\cM}$-action on $F(M)$ via (i), (ii), (iii), is given by: 
{\small
\begin{align*}
\zeta_{F(M)}^{\mathtt{E}_{\cM}}: \mathtt{E}_{\cM}\act F(M) 
\stackrel{\text{Lem}~\textnormal{\ref{lem:coeq-id}}}{\cong} 
\mathtt{E}_{\cM}\act F(\mathtt{E}_{\cM}\act_{\hspace{0.03in}\mathtt{E}_{\cM}} M) 
& \xrightarrow{s^{-1}} F(\mathtt{E}_{\cM} \act \mathtt{E}_{\cM}\act_{\hspace{0.03in}\mathtt{E}_{\cM}} M) 
\cong F((\mathtt{E}_{\cM} \otimes \mathtt{E}_{\cM}) \act_{\hspace{0.03in}\mathtt{E}_{\cM}} M)\\ 
& \xrightarrow{F(m_{\mathtt{E}_{\cM}}\act_{\hspace{0.03in}\mathtt{E}_{\cM}}\id_{M})}
F(\mathtt{E}_{\cM} \act_{\hspace{0.03in}\mathtt{E}_{\cM}} M) \stackrel{\text{Lem}~\textnormal{\ref{lem:coeq-id}}}{\cong} F(M).    
\end{align*}
}

\noindent Finally, considering the isomorphism $\alpha: \mathtt{E}_\cM \act_{\mathtt{E}_\cM} M \xrightarrow{\sim} M$ from $\S$\ref{ss:EMcat}, and  the left $\mathtt{E}_\cM$-module $M$ with action $\beta_M^{\mathtt{E}_\cM}:= \ev_{M,M}\circ(\pi_{\cM}(M)\act\id_M) : \mathtt{E}_\cM \act M \to M$, we get:
\[
\zeta_{F(M)}^{\mathtt{E}_\cM} 
\; \overset{\text{s nat'l}}{=} \; 
F\bigl(\alpha \; (m_{\mathtt{E}_\cM} \act_{\mathtt{E}_\cM} \id_M) \; (\id_{\mathtt{E}_\cM} \act \alpha^{-1})\bigr) s^{-1}_{{\mathtt{E}_\cM}, M}
\; \overset{\textnormal{\eqref{eq:xi-equals-action}}}{=} \; 
F(\beta_M^{\mathtt{E}_\cM}) s^{-1}_{{\mathtt{E}_\cM}, M} \; = \; \xi_{F(M)}^{\mathtt{E}_\cM}.
\]
This completes the proof.
\end{proof}


\subsection{Connections between reflective centers and EM-categories} \label{ss:con-refEM}
Recall the reflective centers $\cE_\cC(\cM)$ from $\S$\ref{sec:refcenter}, the coend objects $\mathtt{C}_\cC$ from $\S$\ref{ss:coendBTC}, and the EM-categories from $\S$\ref{ss:EMcat}. Our main result here connects these constructions as follows.

\begin{proposition} \label{prop:HC} 
The following statements hold.
\begin{enumerate}[\upshape (a)]
    \item We have a functor
\[
H_{\cC}: \cE_{\cC}(\cM) \longrightarrow {\tt C}_{\cC}\text{-}\mathsf{Mod}(\cM), \qquad
(M,  e^M) \mapsto (M, \xi^{{\tt C}_{\cC}}_M), 
\]
where $\xi^{{\tt C}_{\cC}}_M$ is defined by the universal property of ${\tt C}_{\cC}$, for $X \in \cC$: 

\vspace{-.1in}

{\small
\[
\xymatrix@C=5pc@R=1.3pc{
X^* \otimes X \act M
\ar[r]^(.57){\iota_{\cC}(X) \; \act \; \id}
\ar[rd]_{\id \; \otimes \; e^M_{X}}
&
{\tt C}_{\cC} \act M
\ar@{-->}[r]^{\xi^{{\tt C}_{\cC}}_{M}}
& 
M
\\
& 
X^* \otimes X \act M
\ar[ru]_(.6){\ev^L_X \; \act\; \id}
&
}
\]
}
\noindent Here, $H_{\cC}(f) := f$, for all $f:(M, e^M) \to (M', e^{M'}) \in \cE_{\cC}(\cM)$.

\medskip

 \item The functor $H_{\cC}$ above is an equivalence of categories, with quasi-inverse: 
 \[
 \widehat{H}_{\cC}: {\tt C}_{\cC}\text{-}\mathsf{Mod}(\cM) \longrightarrow \cE_{\cC}(\cM), \qquad (M, \xi^{{\tt C}_{\cC}}_M) \mapsto (M,  e^M).
 \]
Here, $e^M$ is defined as follows, for $X \in \cC$: 
{\small
\[
\xymatrix@C=4.5pc{
e^M_X: X \act M 
\ar[r]^(.46){\coev_X^L \; \otimes\; \id\; \act \;\id}
& X \otimes X^* \otimes X \act M 
\ar[r]^(.56){\id \;\otimes \; \iota_{\cC}(X)\; \act\; \id}
& X \otimes {\tt C}_{\cC} \act M 
\ar[r]^(.55){\id \;\otimes \;\xi_M^{{\tt C}_{\cC}}}
& X \act M.
}
\]
}
\noindent Moreover, $\widehat{H}_{\cC}(g) := g$, for all $g:(M, \xi_M^{\mathtt{C}_{\cC}}) \to (M', \xi_{M'}^{\mathtt{C}_{\cC}}) \in {\tt C}_{\cC}\text{-}\mathsf{Mod}(\cM)$.
 \end{enumerate}
\end{proposition}

\begin{proof}
We leave the proof to the reader; detailed computations can be found in the appendix of the ArXiv version 1 of this work.
\end{proof}


\section{Nondegenerate module categories} 
\label{sec:nondeg-main}

In this section, we present the the main result of this article. The conditions of when a module category is considered to be factorizable and be nondegenerate are introduced in $\S$\ref{ss:fmc} and $\S$\ref{ss:nmc}, respectively. Then, it is established that these conditions are equivalent in $\S$\ref{ss:mainthm}. Lastly, we discuss when module categories are, in a sense, weakly factorizable and have trivial symmetric center in $\S$\ref{ss:wf} and $\S$\ref{ss:triv}, respectively; ties to nondegeneracy are examined there as well.

\pagebreak

\begin{hypothesis} \label{hyp:Sec4} 
Throughout the section, the following items are fixed:
\begin{enumerate}[(a)]
\item $\cC:=(\cC, \otimes, \one, c)$ is a braided finite tensor category;
\item $\cM:=(\cM, \triangleright, e)$ is a nonzero, exact, indecomposable, braided finite left $\cC$-module category.
\end{enumerate}
\end{hypothesis}

\begin{example} \label{ex:Creg-hyp}
Given a braided finite tensor category $(\cC, c)$, the regular left $\cC$-module category $\cC_{\text{reg}}$ satisfies Hypothesis~\ref{hyp:Sec4}(b). Namely, $\cC_{\text{reg}}$ is indecomposable and exact by various results in \cite{etingof2015tensor} (see  \cite[Examples~4.29, 4.74(a), 4.90(a)]{walton2024}), and  $\cC_{\text{reg}}$ is braided with $e = c^2$ by Example~\ref{ex:M=C,double}.
\end{example}


\subsection{Factorizable module categories} 
\label{ss:fmc}
Consider the preliminary result below.

\begin{proposition} \label{prop:G-fact}  We have a functor:
\[
G_{\cM}: \cM \; \boxtimes \;  \mathsf{Fun}_{\cC|}(\cM, \cM) \longrightarrow \cE_\cC(\cM), \quad M \boxtimes (F,s) \mapsto (F(M),\; e^{F(M)}),
\]
where the component $e^{F(M)}$ at $X \in \cC$ is defined as:
\begin{equation} \label{eq:G-reflection}
\xymatrix@C=3pc{
e^{F(M)}_X: X \triangleright F(M) 
\ar[r]^(.6){s_{X,M}^{-1}}
& F(X \triangleright M)
\ar[r]^{F(e_{X,M})}
& F(X \triangleright M)
\ar[r]^(.48){s_{X,M}}
&   X \triangleright F(M).
}
\end{equation}
For morphisms $f:M \to M' \in \cM$ and $\phi:(F,s) \Rightarrow (F',s') \in \mathsf{Fun}_{\cC|}(\cM,\cM)$, we define 
\[
G_{\cM}(f \boxtimes \phi):= F'(f) \circ \phi_M: F(M) \to F'(M').
\]
\end{proposition}

\begin{proof}
The functor $G_{\cM}$ is defined on objects because $e^{F(M)}$ satisfies the reflection axiom \eqref{eq:reflection}, due to the braid axiom \eqref{eq:brmod1} and the naturality of the module constraint $s$. Moreover, the functor $G_{\cM}$ is defined on morphisms because $F'(f) \circ \phi_M$ satisfies \eqref{eq:ECMmorp}, due to the naturality of $\phi$, the naturality of $e$, and the naturality of module constraints.
\end{proof}

\begin{definition} \label{def:Mfact}

We call $\cM$ \textit{factorizable} if the functor $G_{\cM}$ above is an equivalence of categories.
\end{definition}

Note that the definition of $\cM$ being factorizable makes sense even when $\cC$ is not finite, or when $\cM$ is decomposable or is not exact. Next, we examine the braided structure of  the functor $G_\cM$.

\begin{proposition}\label{lem:G-fact-2} We have the following statements.
\begin{enumerate}[\upshape (a)]
\item The category $\cM  \boxtimes   \mathsf{Fun}_{\cC|}(\cM, \cM)$ is a braided left $\cC$-module category via:
\[
\begin{array}{rl}
X \act \bigl(M \boxtimes (F,s)\bigr) &:= (X \triangleright M) \boxtimes (F,s)\\[.4pc]
e_{X,M \boxtimes (F,s)} &:= e_{X,M} \; \boxtimes \; \id_{(F,s)}.
\end{array}
\]
for $X,Y \in \cC$, $M \in \cM$, and $(F,s) \in \mathsf{Fun}_{\cC|}(\cM, \cM)$.
\medskip
\item The functor $G:=G_{\cM}$ is a braided left $\cC$-module functor. Here, $s^G_{Y,M \boxtimes (F,s)}$ is defined by
\[
\xymatrix@R=.5pc@C=6pc{
G\bigl(Y \act (M \boxtimes (F,s))\bigr) 
\ar@{=}[d]
\ar@{-->}[r]
& 
Y \act  G\bigl(M \boxtimes (F,s))\bigr)
 \ar@{=}[d]\\
 \bigl(F(Y \act M), \; e^{F(Y \act M)}\bigr) 
 \ar[r]^{s_{Y,M}}
 &
\bigl(Y \act F(M), \; e^{Y \act F(M)}\bigr)
 }
\]
for $Y \in \cC$, $M \in \cM$, and $(F,s) \in \mathsf{Fun}_{\cC|}(\cM, \cM)$.
\end{enumerate}
\end{proposition}

\begin{proof}
Part (a) is straight-forward to check. Part (b) holds by the following computations. First,  $s_{Y,M}$ is a morphism in $\cE_\cC(\cM)$, that is \eqref{eq:ECMmorp} holds, since:
{\small
\begin{align*}
&(\id_X \act s_{Y,M}) \; e^{F(Y \act M)}_X \\[.2pc] 
&\quad \overset{\eqref{eq:G-reflection}}{=} (\id_X \act s_{Y,M}) \;   s_{X,Y \act M} \; \;  F(e_{X, Y \act M}) \; \;  s_{X, Y \act M}^{-1}\\[.2pc] 
&\quad \overset{\eqref{eq:brmod2}}{=} (\id_X \act s_{Y,M}) \;   s_{X,Y \act M}  \; F(c_{Y,X} \act \id_M)  \; F(\id_Y \act e_{X,M})   \;F(c_{X,Y} \act \id_M) \;  s_{X, Y \act M}^{-1}\\[.2pc]
&\quad  \overset{\eqref{eq:mod-s-axioms}, \; s\hspace{0.02in}\text{nat'l}}{=}  (c_{Y,X} \act \id_{F(M)}) \; s_{Y \otimes X, M} \; 
\;F(\id_Y \act e_{X,M}) \; s_{Y \otimes X, M}^{-1}  \;(c_{X,Y} \act \id_{F(M)}) \;  (\id_X \act s_{Y,M})\\[.2pc]
&\quad \overset{\eqref{eq:mod-s-axioms}}{=}  (c_{Y,X} \act \id_{F(M)}) \; (\id_Y \act s_{X,M}) \; s_{Y, X \act M} \;F(\id_Y \act e_{X,M}) \; s_{Y,X \act M}^{-1} \;  (\id_Y \act s_{X,M}^{-1})\\ 
&\quad \; \quad \quad \circ \;   \;(c_{X,Y} \act \id_{F(M)}) \;  (\id_X \act s_{Y,M})\\[.2pc]
&\quad \overset{s\hspace{0.02in}\text{nat'l}}{=}  (c_{Y,X} \act \id_{F(M)}) \; \bigl(\id_Y \act s_{X,M} \;  F(e_{X,M}\bigr) \;  s_{X,M}^{-1}) \;(c_{X,Y} \act \id_{F(M)})  \; (\id_X \act s_{Y,M})\\[.2pc]
&\quad \overset{\eqref{eq:G-reflection}}{=}   (c_{Y,X} \act \id_{F(M)}) \;  (\id_Y \act e_X^{F(M)}) \;(c_{X,Y} \act \id_{F(M)}) \;  (\id_X \act s_{Y,M})\\[.2pc]
&\quad \overset{\eqref{eq:ECMact}}{=} e^{Y \act F(M)}_X  \; (\id_X \act s_{Y,M}).
\end{align*}
}

Next, $s^G$ satisfies \eqref{eq:mod-s-axioms}, \eqref{eq:mod-s-axioms-2} as $s$ does. Thus, $(G,s^G)$ is a left $\cC$-module functor. Now, $(G,s^G)$ is a braided module functor, that is, \eqref{eq:braidmodfunc} holds, since:
{\small
\[
\begin{array}{rl}
e_{X}^{G\bigl(M \boxtimes (F,s)\bigr)} \circ  s^G_{X,M \boxtimes (F,s)} 
& = \; e_{X}^{F(M)} \circ  s_{X,M}  \; \; \overset{\eqref{eq:G-reflection}}{=} \;  s_{X,M}\;  \circ  F(e_{X,M})\\[.4pc]
&= \; s^G_{X,M \boxtimes (F,s)} \circ   G\bigl(e_{X,M} \boxtimes \id_{(F,s)}\bigr)
\; \; = \; s^G_{X,M \boxtimes (F,s)} \circ  G\bigl(e_{X,M \boxtimes (F,s)}\bigr)
\end{array}
\]
}
This completes the proof.
\end{proof}

Next, we examine when regular braided module categories are factorizable; see Example~\ref{ex:Creg-hyp}.

\begin{example} \label{ex:regular-fact}
Take $\cM$ to be the regular left $\cC$-module category $\cC_{\textnormal{reg}}$, which is braided via $e:= c^2$ [Example~\ref{ex:Creg-hyp}].  We will show that the functor below is the functor $G_\cC$ from $\S$\ref{subsec:nondeg-monoidal}. By its definition, the functor below will be an equivalence of categories precisely when $G_{\cC_{\textnormal{reg}}}$ from Proposition~\ref{prop:G-fact} is an equivalence of categories.
{\small
\[
\xymatrix@C=2pc@R=0pc{
\cC \boxtimes \cC^{\otimes \textnormal{op}} 
\ar[r]^(.33){\eqref{eq:FunCop}}_(.33){\sim}
&
\cC_{\textnormal{reg}} \; \boxtimes \; \mathsf{Fun}_{\cC|}(\cC_{\textnormal{reg}}, \cC_{\textnormal{reg}}) 
\ar[r]^(.62){G_{\cC_{\textnormal{reg}}}}
&
\cE_{\cC}(\cC_{\textnormal{reg}})
\ar[r]^(.5){[\textnormal{Rem}.\; 
 \textnormal{\ref{rem:Ec(C)Z(C)}}]}_(.5){\sim}
&
\cZ(\cC)\\
V \boxtimes W
\ar@{|->}[r]
&
V \boxtimes (- \otimes W)
\ar@{|->}[r]
&
(V \otimes W, \; e^{V \otimes W})
\ar@{|->}[r]
&
(V \otimes W, \; c_{V \otimes W, X}^{-1} \circ e_X^{V \otimes W}).
}
\] 
}

\noindent Now the functor above is  equal to the functor $G_\cC$ from $\S$\ref{subsec:nondeg-monoidal} due to the computations below:
{\small
\[
\begin{array}{rl}
c^{V \otimes W}_X:= c_{V \otimes W, X}^{-1} \circ e_X^{V \otimes W}  &= \; 
c_{V \otimes W, X}^{-1} \; (c_{V,X} \otimes \id_W) \; (c_{X,V} \otimes \id_W)\\[.4pc]
 &
 \overset{\eqref{eq:braid1}}{=} \; 
(\id_V \otimes c_{W,X}^{-1}) \; (c_{V,X}^{-1} \otimes \id_W) \; (c_{V,X} \otimes \id_W) \; (c_{X,V} \otimes \id_W)\\[.4pc]
 &= \; 
(\id_V \otimes c_{W,X}^{-1}) \; (c_{X,V} \otimes \id_W).
\end{array}
\]
}

\noindent Thus, $(\cC_{\textnormal{reg}},\;  e:=c^2)$ is factorizable as a braided left $\cC$-module category if and only if $(\cC, c)$ is factorizable as a braided finite tensor category.
\end{example}


\subsection{Nondegenerate module categories}
\label{ss:nmc}

Recall the notation from $\S\S$\ref{ss:coendBTC}, \ref{ss:Endmod}.

\begin{definition} \label{def:omegaM}  Using the universal properties of $\mathtt{E}_{\cC}$ and $\mathtt{E}_{\cM}$, define the following morphism in $\cC$,
\[
\omega_{\cM}: \one \to \mathtt{E}_{\cC} \otimes \mathtt{E}_{\cM},
\]
to be the unique morphism that completes the commutative diagram below, for each $X \in \cC$, $M \in \cM$.
{\small
\[
\hspace{-.02in}
\xymatrix@C=5.3pc@R=1pc{ 
\one 
\ar@{-->}[r]^(.45){\omega_{\cM}}
\ar[d]_(.4){\coev^L_X\; \otimes \; \textnormal{coev}_{\one, M}}
& 
\mathtt{E}_{\cC} \otimes \mathtt{E}_{\cM}
\ar[r]^(.38){\pi_{\cC}(X)\;  \otimes \; \pi_{\cM}(M)}
& 
X \otimes X^* \otimes \; \underline{\textnormal{Hom}}(M,M)\\
X \otimes X^* \otimes \; \underline{\textnormal{Hom}}(M,M) 
\ar[r]^{\sim}_{\eqref{eq:internalHom-2}}
&
X  \otimes \; \underline{\textnormal{Hom}}(M,X^* \triangleright M)
\ar[r]^{\sim}_{\id \; \otimes \; \underline{\textnormal{Hom}}(\id, e_{X^{\scalebox{.6}{$*$}},M})}
& 
X  \otimes \; \underline{\textnormal{Hom}}(M,X^* \triangleright M)
\ar[u]^{\sim}_{\eqref{eq:internalHom-2}}
}
\]
}
We refer to $\omega_{\cM}$ as the {\it universal copairing of $\cC$ and $\cM$}.
\end{definition}

With the morphism $\omega_{\cM}$ above, we define nondegenerate module categories as follows.

\begin{definition} \label{def:Mnondeg} 
We say that $\cM$ is {\it nondegenerate}  if

\vspace{-.2in}
\[
\xymatrix@C=5pc{
\theta_{\cM}: \mathtt{E}_{\cC}^*
\ar[r]^(.43){\id \; \otimes \; \omega_{\cM}}
&
\mathtt{E}_{\cC}^* \otimes \mathtt{E}_{\cC} \otimes \mathtt{E}_{\cM}
\ar[r]^(.6){\ev^L_{\mathtt{E}_{\cC}} \; \otimes \; \id}
& 
\mathtt{E}_{\cM}
} 
\]
is an isomorphism in $\cC$.
\end{definition}

Next, we study when regular braided module categories [Example~\ref{ex:Creg-hyp}] are  nondegenerate. 

\begin{example} \label{ex:Mreg-nondeg}
Take $\cM$ to be the regular left $\cC$-module category $\cC_{\textnormal{reg}}$, which is braided via $e:= c^2$. In this case, $\underline{\textnormal{Hom}}(Y,Y) \cong Y \otimes Y^*$, and $\textnormal{coev}_{\one,Y} = \coev_Y$, for $Y \in \cC_{\textnormal{reg}}$. Now by incorporating \eqref{eq:internalHom-2}, the nondegeneracy of $\cC_{\textnormal{reg}}$ depends on a morphism, $\omega_{\cC_{\textnormal{reg}}}: \one \to {\tt E}_\cC \otimes {\tt E}_{\cC_{\textnormal{reg}}}$, defined by:

\vspace{-.1in}

\begin{equation}\label{eq:omegaCreg}
(\pi_{\cC}(X) \otimes \pi_{\cC_{\textnormal{reg}}}(Y)) \circ \omega_{\cC_{\textnormal{reg}}} 
= \bigl(\id_X \otimes (c_{Y,X^*} \circ c_{X^*,Y}) \otimes \id_{Y^*}\bigr) \circ (\coev_X \otimes \coev_Y).
\end{equation}

\vspace{.1in}

We claim that $(\omega_{\cC_{\textnormal{reg}}})^*$ is the morphism $\omega_\cC$ from $\S\S$\ref{ss:coendBTC}, \ref{subsec:nondeg-monoidal}. 
Note that $(\coev^L_Z)^* = \ev^L_{Z^*}$, and $(\ev^L_Z)^* = \coev^L_{Z^*}$, and $(c_{Z,W})^* = c_{W^*,Z^*}$,
for  $W,Z \in \cC$.  Now, with \eqref{eq:EcCc2}, dualizing \eqref{eq:omegaCreg} yields:

\vspace{-.1in}

\[
(\omega_{\cC_{\textnormal{reg}}})^* \circ (\iota_{\cC}(Y^*) \otimes \iota_{\cC}(X^*)) 
= (\ev^L_{Y^*} \otimes \ev^L_{X^*}) \circ \bigl(\id_{Y^{**}} \otimes (c_{X^{**},Y^*} \circ c_{Y^*,X^{**}}) \otimes \id_{X^*}\bigr).
\] 
By the uniqueness of $\omega_\cC$,  defined in Lemma~\ref{lem:EcHopfAlg}(b), we get that $(\omega_{\cC_{\textnormal{reg}}})^* = \omega_\cC$ as claimed.

Next, $\cC_{\textnormal{reg}}$ is nondegenerate as a braided left $\cC$-module category if and only if $\theta_{\cC_{\textnormal{reg}}}$ from Definition~\ref{def:Mnondeg} is an isomorphism. Recall the morphism $\theta_{\cC}$ from $\S$\ref{subsec:nondeg-monoidal} attached to the nondegeneracy of $\cC$. Now with \eqref{eq:EcCc}, we obtain that:
\[
(\theta_{\cC_{\textnormal{reg}}})^* \; :=\;  
\bigl((\omega_{\cC_\textnormal{reg}})^* \otimes \id_{{\tt E}_\cC^{**}}\bigr)
\bigl(\id_{{\tt E}_{\cC_\textnormal{reg}}^*} \otimes (\ev^L_{{\tt E}_\cC})^*\bigr) 
\; = \; 
\bigl(\omega_{\cC} \otimes \id_{{\tt C}_\cC^{*}}\bigr)
\bigl(\id_{{\tt C}_{\cC}} \otimes \coev^L_{{\tt C}_\cC}\bigr) \; =: \; \theta_{\cC}.
\]
This implies that  $(\cC_{\textnormal{reg}},\;  e:=c^2)$ is  nondegenerate as a braided left $\cC$-module category precisely when $(\cC, c)$ is nondegenerate as a braided finite tensor category.
\end{example}


\subsection{Equivalence of nondegeneracy and factorizability} 
\label{ss:mainthm}
The main result of the article is the following theorem.  

\begin{theorem} \label{thm:FactNondeg} Recall Hypothesis~\ref{hyp:Sec4}. We obtain that the left $\cC$-module category $\cM$ is factorizable  if and only if $\cM$ is nondegenerate as a left $\cC$-module category.  
\end{theorem}

We need a technical lemma before we establish the result above; the proof is left to the reader.

\begin{lemma} \label{lem:prelim}
Take a functor $F: \cM \to \cM$,  objects $Z_1, Z_2 \in \cC$ and $M \in \cM$, and morphisms $h: Z_1 \to Z_1 \otimes Z_2$ in $\cC$ and $\ell: Z_2 \act M \to M$ in $\cM$. Then, we have the equality of morphisms: 
\[
\bigl(\id_{Z_1} \otimes  F(\ell)\bigr) \; \bigl(\id_{Z_1} \otimes  s^{-1}_{Z_2,M}\bigr) \;  \bigl(h \act \id_{F(M)}\bigr)
\; = \; s_{Z_1,M} \;  F\bigl((\id_{Z_1} \otimes  \ell)(h \act \id_M)\bigr) \;  s^{-1}_{Z_1,M},
\]
which go from $Z_1 \otimes F(M)$ to $Z_1 \otimes F(M)$. \qed
\end{lemma}

\medskip

This brings us to the proof of the main result of this article.

\begin{proof}[Proof of Theorem~\ref{thm:FactNondeg}]
Recall the category equivalences, $H_{\cM}$ and $H_{\cC}$, from Propositions~\ref{prop:HM} and~\ref{prop:HC} respectively, the quasi-inverse $\widehat{H}_\cC$ of $H_{\cC}$ given in Proposition~\ref{prop:HC}(b), and also the functor $\textnormal{Res}_{\theta_{\cM}}$ from Lemma~\ref{lem:AlgResEquiv}. We will show that the following functorial diagram commutes.
\[
\xymatrix@C=7pc@R=1.2pc{
\cM \; \boxtimes \;  \mathsf{Fun}_{\cC|}(\cM, \cM)
\ar[r]^(.6){G_{\cM}}
\ar[d]_(.45){H_{\cM}}^{\sim}
&
\cE_{\cC}(\cM)\\
{\tt E}_{\cM}\text{-}\mathsf{Mod}(\cM)
\ar[r]^(.4){\textnormal{Res}_{\hbox{\tiny$\theta_{\cM}$}}}
&
{\tt E}_{\cC}^*\text{-}\mathsf{Mod}(\cM) \overset{\eqref{eq:EcCc}}{\cong} {\tt C}_{\cC}\text{-}\mathsf{Mod}(\cM) \ar[u]_(.6){\widehat{H}_{\cC}}^(.5){\sim}
}
\]

\smallskip

Consider the following computation:
\smallskip
{\small
\[
\begin{array}{l}
 \bigl(\widehat{H}_{\cC} \circ \textnormal{Res}_{\hbox{\tiny$\theta_{\cM}$}} \circ H_{\cM}\bigr)\bigl(M \boxtimes (F,s)\bigr) 
\; = \; \bigl(\widehat{H}_{\cC} \circ \textnormal{Res}_{\hbox{\tiny$\theta_{\cM}$}}\bigr)\bigl(F(M), \; \xi_{F(M)}^{\mathtt{E}_{\cM}}\bigr) 
\; = \; \widehat{H}_{\cC}\bigl(F(M), \; \xi_{F(M)}^{\mathtt{E}_{\cM}}(\theta_\cM \act \id_{F(M)}) \bigr)\\[.6pc]
= \Bigl(F(M), \; e^{F(M)}:=\Bigl\{\bigl(\id_X \otimes \xi_{F(M)}^{\mathtt{E}_{\cM}}(\theta_\cM \act \id_{F(M)})\bigr)\bigl(\id_X \otimes \iota_{\cC}(X) \act \id_{F(M)} \bigr)\bigl(\coev_X^L \otimes \id_X \act \id_{F(M)} \bigr)\Bigr\}_{X \in \cC} \Bigr). \\
\end{array}
\]
}

\noindent To show that $\widehat{H}_{\cC} \circ \textnormal{Res}_{\hbox{\tiny$\theta_{\cM}$}} \circ H_{\cM}$ is equal to $G_{\cM}$, it suffices to show that the morphism $e^{F(M)}_X$ above is equal to the morphism $e^{F(M)}_X$ in \eqref{eq:G-reflection}.  This is done below; here ``level ex." denotes level exchange:
\smallskip
{\small
\begin{align*}
&\bigl(\id_X \otimes \xi_{F(M)}^{\mathtt{E}_{\cM}}(\theta_\cM \act \id_{F(M)})\bigr)\bigl(\id_X \otimes \iota_{\cC}(X) \act \id_{F(M)} \bigr)\bigl(\coev_X^L \otimes \id_X \act \id_{F(M)} \bigr)\\[.4pc]
&\overset{\text{Def.}~\textnormal{\ref{def:Mnondeg}}}{=} \bigl(\id_X \otimes \xi_{F(M)}^{\mathtt{E}_{\cM}}((\ev_{\mathtt{E}_{\cC}}^L \otimes \id_{\mathtt{E}_{\cM}})(\id_{\mathtt{E}_{\cC}^*} \otimes  \omega_{\cM}) \act \id_{F(M)})\bigr)\bigl(\id_X \otimes \iota_{\cC}(X) \act \id_{F(M)} \bigr)\bigl(\coev_X^L \otimes \id_X \act \id_{F(M)} \bigr)\\[.4pc]
&\overset{\text{Prop.}~\textnormal{\ref{prop:HM}}, ~\text{s nat'l}}{=} 
\bigl(\id_X \otimes F(\ev_{M,M})\; s^{-1}_{\underline{\text{Hom}}(M,M),M} \; (\pi_{\cM}(M) \act \id_{F(M)})\bigr)
\bigl(\id_X \otimes ((\ev_{\mathtt{E}_{\cC}}^L \otimes \id_{\mathtt{E}_{\cM}}) \act \id_{F(M)})\bigr)\\[.2pc]
&\hspace{.6in} \circ \; 
\bigl(\id_X \otimes ((\id_{\mathtt{E}_{\cC}^*} \otimes  \omega_{\cM}) \act \id_{F(M)})\bigr)
\bigl(\id_X \otimes \iota_{\cC}(X) \act \id_{F(M)} \bigr)\bigl(\coev_X^L \otimes \id_X \act \id_{F(M)} \bigr)\\[.4pc]
&\overset{\text{level ex.}}{=} 
\bigl(\id_X \otimes F(\ev_{M,M})\; s^{-1}_{\underline{\text{Hom}}(M,M),M} \; (\pi_{\cM}(M) \act \id_{F(M)})\bigr)
\bigl(\id_X \otimes \ev_{\mathtt{E}_{\cC}}^L \otimes \id_{\mathtt{E}_{\cM}} \act \id_{F(M)})\bigr)
\\[.2pc]
&\hspace{.6in} \circ \; 
\bigl(\id_X \otimes \iota_{\cC}(X) \otimes \id_{\mathtt{E}_{\cC} \otimes \mathtt{E}_{\cM}} \act \id_{F(M)} \bigr)
\bigl(\id_{X \otimes X^* \otimes X} \otimes \omega_{\cM} \act \id_{F(M)})\bigr)
\bigl(\coev_X^L \otimes \id_X \act \id_{F(M)} \bigr)\\[.4pc]
&\overset{\eqref{eq:dualmor}, \eqref{eq:EcCc}, \eqref{eq:EcCc2}'}{=} 
\bigl(\id_X \otimes F(\ev_{M,M})\; s^{-1}_{\underline{\text{Hom}}(M,M),M} \; (\pi_{\cM}(M) \act \id_{F(M)})\bigr)
\\[.2pc]
&\hspace{.6in} \circ \; \bigl(\id_X \otimes \ev_X^L(\id_{X^*} \otimes \ev_X^R \otimes \id_X) \otimes \id_{\mathtt{E}_{\cM}} \act \id_{F(M)})\bigr)
\\[.2pc]
&\hspace{.6in} \circ \; 
\bigl(\id_{X \otimes X^* \otimes X} \otimes {}^* \hspace{-.01in} \iota_{\cC}(X) \otimes \id_{\mathtt{E}_{\cM}} \act \id_{F(M)} \bigr)
\bigl(\id_{X \otimes X^* \otimes X} \otimes \omega_{\cM} \act \id_{F(M)})\bigr)
\bigl(\coev_X^L \otimes \id_X \act \id_{F(M)} \bigr)\\[.4pc]
&=
\bigl(\id_X \otimes F(\ev_{M,M})\bigr)\; \bigl(\id_X \otimes  s^{-1}_{\underline{\text{Hom}}(M,M),M}\bigr) \; \bigl(h \act \id_{F(M)}\bigr), 
\end{align*}
}

\vspace{0.1in}

\noindent  for the morphism $h$ defined as:
{\small
\begin{align*}
&h:= \Bigl(\id_X  \otimes  \pi_{\cM}(M) \;  \bigl(\ev_X^L(\id_{X^*} \otimes \ev_X^R \otimes \id_X) \otimes \id_{\mathtt{E}_{\cM}}\bigr) \; \bigl(\id_{X^* \otimes X} \otimes {}^* \hspace{-.01in} \iota_{\cC}(X) \otimes \id_{\mathtt{E}_{\cM}}\bigr)\;  
\bigl( \id_{X^* \otimes X} \otimes \omega_{\cM} \bigr) \Bigr)\\
& \hspace{.6in} \circ \bigl(\coev_X^L \otimes \id_X \bigr).
\end{align*}
}

\vspace{0.1in}

Now using Lemma~\ref{lem:prelim}, with $Z_1 = X$, and $Z_2 = \underline{\text{Hom}}(M,M)$, and $\ell = \ev_{M,M}$, we get that 
\smallskip
{\small
\begin{align*}
&\bigl(\id_X \otimes \xi_{F(M)}^{\mathtt{E}_{\cM}}(\theta_\cM \act \id_{F(M)})\bigr)\bigl(\id_X \otimes \iota_{\cC}(X) \act \id_{F(M)} \bigr)\bigl(\coev_X^L \otimes \id_X \act \id_{F(M)} \bigr)\\[.2pc]
&\overset{\text{Lem.}~\textnormal{\ref{lem:prelim}}}{=}
s_{X,M} \;  F\bigl((\id_{X} \otimes  \ev_{M,M})(h \act \id_M)\bigr) \;  s^{-1}_{X,M}.
\end{align*}
}
Next, consider the following computation:
{\small
\begin{align*}
&(\id_{X} \otimes  \ev_{M,M})(h \act \id_M)\\[.4pc]
& \overset{\text{level\hspace{0.02in}ex.}, \hspace{0.04in} \text{Def.}\textnormal{\ref{def:Mnondeg}}, \hspace{0.04in} \eqref{eq:EcCc2}'}{=}
\bigl(\id_{X} \otimes \ev_X^L \act \id_M \bigr)
\bigl(\id_{X \otimes X^* \otimes X} \otimes \ev_{M,M}\bigr)
\bigl(\id_{X \otimes X^*}  \otimes \underline{\text{Hom}}(\id_M, e_{X,M}) \act \id_M \bigr)
\\[.2pc]
&\hspace{.9in} \circ \;
\bigl(\id_{X \otimes X^*} \otimes \ev_X^R \otimes \id_{X \otimes \underline{\text{Hom}}(M,M)} \act \id_M \bigr)
 \bigl(\id_{X \otimes X^* \otimes X} \otimes \coev_{{}^* \hspace{-.02in} X}^L \otimes \coev_{\one,M} \act \id_M \bigr)
 \\[.2pc]
&\hspace{.9in} \circ \;
\bigl(\coev_X^L \otimes \id_X \act \id_M\bigr)\\[.4pc]
& \overset{\ev~\text{nat'l},\hspace{0.04in} \eqref{eq:internalHom},\hspace{0.04in}\eqref{eq:internalHom-2}}{=}
\bigl(\id_{X} \otimes \ev_X^L \act \id_M \bigr)
\bigl(\id_{X \otimes X^*} \otimes \ev_{M,X \act M}\bigr)
\bigl(\id_{X \otimes X^*}  \otimes \underline{\text{Hom}}(\id_M, e_{X,M}) \act \id_M \bigr)
\\[.2pc]
&\hspace{.9in} \circ \;
\bigl(\id_{X \otimes X^*} \otimes \ev_X^R \otimes \id_{X \otimes \underline{\text{Hom}}(M,M)} \act \id_M \bigr)
 \bigl(\id_{X \otimes X^* \otimes X} \otimes \coev_{{}^* \hspace{-.02in} X}^L \otimes \coev_{\one,M} \act \id_M \bigr)
 \\[.2pc]
&\hspace{.9in} \circ \;
\bigl(\coev_X^L \otimes \id_X \act \id_M\bigr)\\[.4pc]
& \overset{\text{ev nat'l}, \hspace{0.04in}\eqref{eq:internalHom-2}}{=}
\bigl(\id_{X} \otimes \ev_X^L \act \id_M \bigr)
\bigl(\id_{X \otimes X^*}  \otimes e_{X,M} \bigr)
\bigl(\id_{X \otimes X^* \otimes X} \otimes \ev_{M,M}\bigr)
\\[.2pc]
&\hspace{.9in} \circ \;
\bigl(\id_{X \otimes X^*} \otimes \ev_X^R \otimes \id_{X \otimes \underline{\text{Hom}}(M,M)} \act \id_M \bigr)
 \bigl(\id_{X \otimes X^* \otimes X} \otimes \coev_{X}^R \otimes \coev_{\one,M} \act \id_M \bigr)
 \\[.2pc]
&\hspace{.9in} \circ \;
\bigl(\coev_X^L \otimes \id_X \act \id_M\bigr)\\[.4pc]
& \overset{\text{rigidity}}{=}
\bigl(\id_{X} \otimes \ev_X^L \act \id_M \bigr)
\bigl( \id_{X \otimes X^*} \otimes e_{X,M} \bigr)
\bigl(\id_{X \otimes X^* \otimes X} \otimes \ev_{M,M}\bigr)
\bigl(\id_{X \otimes X^* \otimes X}  \otimes \coev_{\one,M} \act \id_M \bigr)
 \\[.2pc]
&\hspace{.9in} \circ \;
\bigl(\coev_X^L \otimes \id_X \act \id_M\bigr)\\
& \overset{\eqref{eq:evMMcoev1M}}{=}
\bigl(\id_{X} \otimes \ev_X^L \act \id_M \bigr)
\bigl( \id_{X \otimes X^*} \otimes e_{X,M} \bigr)
\bigl(\coev_X^L \otimes \id_X \act \id_M\bigr)\\[.2pc]
& \overset{\text{rigidity, level ex.}}{=}
e_{X,M}.
\end{align*}
}
By combining the results above, we achieve the following result:
\smallskip
{\small
\begin{align*}
\bigl(\id_X \otimes \xi_{F(M)}^{\mathtt{E}_{\cM}}(\theta_\cM \act \id_{F(M)})\bigr)\bigl(\id_X \otimes \iota_{\cC}(X) \act \id_{F(M)} \bigr)\bigl(\coev_X^L \otimes \id_X \act \id_{F(M)} \bigr)
\; = \;
s_{X,M} \;  F(e_{X,M}) \;  s^{-1}_{X,M},
\end{align*}
}
\noindent which shows that $\widehat{H}_{\cC} \circ \textnormal{Res}_{\hbox{\tiny$\theta_{\cM}$}} \circ H_{\cM} = G_{\cM}$ on objects. This also extends to equality on morphisms.

\smallskip

So, $G_{\cM}$ is an equivalence of categories if and only if $\textnormal{Res}_{\theta_{\cM}}$ is an equivalence of categories, which, in turn, occurs if and only if $\theta_{\cM}$ is an isomorphism of algebras [Lemma~\ref{lem:AlgResEquiv}]. Thus, the  $\cC$-module category $\cM$ is factorizable precisely when it is  nondegenerate; see Definitions~\ref{def:Mfact} and~\ref{def:Mnondeg}. 
\end{proof}


\subsection{On weak factorizability}
\label{ss:wf}
Recall Hypothesis~\ref{hyp:Sec4} and  the copairing $\omega_{\cM}: \one \to \mathtt{E}_{\cC} \otimes \mathtt{E}_{\cM}$ from Definition~\ref{def:omegaM}. Consider the terminology below.

\begin{definition} \label{def:wf}
We say that $\cM$ is {\it weakly factorizable} if the following linear map is bijective:
\begin{equation*}
  \Omega_{\cM} : \Hom_{\cC}(\mathtt{E}_{\cC},\one) \rightarrow \Hom_{\cC}(\one, \mathtt{E}_{\cM}), \qquad f\mapsto (f\otimes \id_{{\tt E}_{\cM}}) \circ \omega_{\cM}.
\end{equation*}
\end{definition}

\begin{example}
The regular braided left $\cC$-module category $\cC_{\text{reg}}$ from Example~\ref{ex:M=C,double} is weakly factorizable precisely when $\cC$ is weakly factorizable as a braided finite tensor category. Indeed, $(\omega_{\cC_{\text{reg}}})^* = \omega_\cC$ as shown in Example~\ref{ex:Mreg-nondeg}, and $\mathtt{E}_{\cC}^* \cong \mathtt{C}_{\cC}$ by \eqref{eq:EcCc}. Now, it is straightforward to show,  for the map $\Omega_{\cC}$ from $\S$\ref{subsec:nondeg-monoidal}, that the diagram below commutes.
\[
\xymatrix@R=1.5pc@C=4pc{
\Hom_{\cC}(\mathtt{E}_{\cC}, \one) 
\ar[r]^(.48){\Omega_{\cC_{\text{reg}}}}
\ar[d]_{(-)^*}
& \Hom_{\cC}(\one, \mathtt{E}_{\cC_{\text{reg}}})
\ar[d]^{(-)^*}\\
\Hom_{\cC}(\one, \mathtt{C}_{\cC})
\ar[r]^{\Omega_{\cC}}
&\Hom_{\cC}(\mathtt{C}_{\cC}, \one)
}
\]
So,  $\Omega_{\cC_{\text{reg}}}$ is bijective precisely when the map $\Omega_{\cC}$  is bijective, as required.
\end{example}

Now we tie weak factorizability to nondegeneracy below.

\begin{lemma} \label{lem:non-wf}
If $\cM$ is nondegenerate,  then it is weakly factorizable.
\end{lemma}

\begin{proof}
Since $\cM$ is nondegenerate, we have the  isomorphism $\theta_{\cM}: \mathtt{E}_{\cC}^* \to \mathtt{E}_{\cM}$ from Definition~\ref{def:Mnondeg}. It is straightforward to verify that the inverse of $\Omega_{\cM}$ is given by:
\[
\Omega_{\cM}^{-1} :  \Hom_{\cC}(\one, \mathtt{E}_{\cM}) \rightarrow  \Hom_{\cC}(\mathtt{E}_{\cC},\one), \qquad g \mapsto \text{ev}_{{\mathtt E}_\cC}  
\bigl( (\theta_{\cM}^{-1} \circ g) \otimes \id_{{\mathtt E}_\cC} \bigr)
\]
in this case. Thus, $\cM$ is weakly factorizable.
\end{proof}

Note that  weak factorizability is tied to the invertibility of {\it S-matrices} of braided fusion categories; see \cite[$\S$5.1]{shimizu2019non} and references within. Moreover, $S$-matrices for braided module categories  were introduced in the semisimple case in the work of Johnson-Freyd--Reutter; see the proof of \cite[Theorem~2.57]{johnson2024minimal}. So, we pose the following question.

\begin{question} \label{ques:JFR}
What is the connection between the weak factorizability condition here and the invertibility of $S$-matrices for braided module categories  in the semisimple case?
\end{question}


\subsection{On trivializability}
\label{ss:triv}
Recall Hypothesis~\ref{hyp:Sec4}, and consider the  terminology below.

\begin{definition} \label{def:trivM}
The {\it symmetric center} of $\cM$, denoted by $\cZ_{2}(\cM)$, is defined to be the full subcategory of $\cC$  consisting of the objects:
\[
\text{Ob}(\cZ_2(\cM)) := \{ X \in \cC \; | \; e_{X,M}=\id_{X\act M}, \; \forall \, M\in\cM \}.
\]
We say that $\cM$ has {\it trivial symmetric center} if $\cZ_2(\cM) \simeq \mathsf{FdVec}$.
\end{definition}

For instance, the regular braided left $\cC$-module category $(\cC_{\reg}, e = c^2)$ [Example~\ref{ex:M=C,double}] has trivial symmetric center precisely when the braided finite tensor category $\cC$ has trivial symmetric center.

\smallskip

We obtain the preliminary result below.

\begin{lemma}\label{lem:Z2CM-tensor-subcat}
The symmetric center  $\cZ_{2}(\cM)$ is a finite tensor subcategory of $\cC$.
\end{lemma}

\begin{proof}
By \cite[Lemma~4.4]{shimizu2019non}, it suffices to show that $\cZ_{2}(\cM)$ is closed under monoidal product and duals. Take objects $X,Y\in\cZ_{2}(\cM)$. Then, for $M\in\cM$, we have:
{\small
\begin{align*}
  e_{X\otimes Y,M} & \; \stackrel{\eqref{eq:brmod1}}{=}\;  (\id_X \act e_{Y,M}) (c_{Y,X} \act \id_M)  (\id_Y \act e_{X,M})(c_{Y,X}^{-1} \act \id_M)
  \\[.2pc]
  & \; = \;  (c_{Y,X} \act \id_M)(c_{Y,X}^{-1} \act \id_M) \; =\;  \id_{(X\otimes Y)\act M}.
\end{align*}
}
Thus, $X\otimes Y \in \cZ_{2}(\cM)$. We also compute:
{\small
\begin{align*}
  e_{X^*,M} & \;  \overset{\text{rigidity}}{=}
  e_{X^*,M} \bigl(\ev^L_X\otimes\id_{X^*} \act \id_M \bigr) \bigl(\id_{X^*}\otimes\coev^L_X \act \id_M \bigr) 
  \\[.4pc]
   & \;  = 
  e_{X^*,M} \bigl(\ev^L_X\otimes\id_{X^*} \act \id_M \bigr) \bigl(\id_{X^*} \otimes (c_{X^*,X} \circ c^{-1}_{X^*,X}) \act\id_M \bigr)\bigl(\id_{X^*}\otimes\coev^L_X \act \id_M \bigr) 
  \\[.4pc]
   & \; \overset{X \in \cZ_2(\cM), \, \text{level ex.}}{=} 
     \bigl(\ev^L_X \otimes \id_{X^*} \act \id_M \bigr) 
  \bigl(\id_{X^*} \otimes \id_X \otimes e_{X^*,M} \bigr) \bigl(\id_{X^*} \otimes c_{X^*,X}\act\id_M \bigr) 
  \\[.2pc]
   & \; \qquad \; \qquad \; \qquad  \circ \bigl(\id_{X^*}\otimes\id_{X^*}\otimes e_{X,M} \bigr)
   \bigl(\id_{X^*} \otimes c_{X^*,X}^{-1} \act \id_M \bigr) \bigl(\id_{X^*}\otimes\coev^L_X \act \id_M \bigr)  
  \\[.2pc]
  & \overset{\eqref{eq:brmod1}}{=}
  \bigl(\ev^L_X\otimes\id_{X^*}\act\id_M \bigr) \bigl(\id_{X^*} \otimes e_{X\otimes X^*,M}\bigr)
  \bigl(\id_{X^*}\otimes\coev^L_X \act \id_M \bigr) 
  \\[.2pc]
   & \; \overset{\text{$e$\,nat'l}, \, \eqref{eq:braidmod-unit}}{=} \; \bigl(\ev^L_X\otimes\id_{X^*} \act \id_M \bigr)
   \bigl(\id_{X^*}\otimes\coev^L_X \act \id_M \bigr) 
  \\
   & \;  \overset{\text{rigidity}}{=} \; \id_{X^*\act M}.
\end{align*}
}

\noindent  Thus, $X^*\in\cZ_{2}(\cM)$.
With a similar argument, we also have closure under right duals.
\end{proof}

We also need the technical lemma below pertaining to the end algebras introduced in $\S\S$\ref{ss:coendBTC},~\ref{ss:Endmod}. The techniques in the next two results are inspired by \cite[$\S$5.3]{shimizu2019non}.

\begin{lemma} \label{lem:Z2M-u}
Denote $\cD:= \cZ_2(\cM)$, and consider the canonical morphism $\phi: \mathtt{E}_{\cC} \to \mathtt{E}_{\cD}$ from the universal property of ends (namely, $\pi_\cD(X) \circ \phi = \pi_\cC(X)$ for each $X \in \cD$). Then:
\begin{enumerate}[\upshape (a)]
\item $\phi$ is an epimorphism, 

\smallskip

\item $\phi$ is unital, that is, $\phi \circ u_{\mathtt{E}_{\cC}} = u_{\mathtt{E}_{\cD}}$, and 

\smallskip

\item $\bigl(\phi\otimes\id_{\mathtt{E}_{\cM}} \bigr) \, \omega_\cM  = 
\bigl(\phi \otimes \id_{\mathtt{E}_{\cM}}\bigr) \bigl (u_{\mathtt{E}_{\cC}} \otimes u_{\mathtt{E}_{\cM}} \bigr)$.
\end{enumerate}
\end{lemma}

\begin{proof}
We employ Lemma~\ref{lem:Z2CM-tensor-subcat}, treating $\cD$ as a finite tensor category. 

\smallskip

(a) This holds by \cite[Lemma~4.9]{shimizu2017character}.

\smallskip

(b) By taking duals, it suffices to show that $u_{\mathtt{E}_{\cC}}^* \circ \phi^* = u_{\mathtt{E}_{\cD}}^*$.  This holds as, for any $X \in \cD$:
\[
u_{\mathtt{E}_{\cC}}^* \; \phi^* \; \iota_{\cD}(X^*) 
\, = \,
\varepsilon_{\mathtt{C}_{\cC}} \; \phi^* \; \iota_{\cD}(X^*) 
\,= \, 
\varepsilon_{\mathtt{C}_{\cC}}\;  \iota_{\cC}(X^*) 
\, = \,
\ev^L_{X^*}
\, = \,
\varepsilon_{\mathtt{C}_{\cD}} \;  \iota_{\cD}(X^*) 
\, = \,
u_{\mathtt{E}_{\cD}}^* \;  \iota_{\cD}(X^*),  
\]
where these equalities follow from Lemma~\ref{lem:EcHopfAlg}(a),~\eqref{eq:EcCc}, and~\eqref{eq:EcCc2}.

\smallskip

(c) This holds, since for $X\in\cD$ and $M\in\cM$, the following diagram commutes.

\vspace{-.1in}

{\small
\[
\xymatrix@C=1pc@R=1pc{
\one 
\ar[rr]^{\omega_{\cM}}
\ar[d]_{u_{\mathtt{E}_{\cC}} \, \otimes  \, u_{\mathtt{E}_{\cM}}}
\ar@/^1.4pc/[rdd]^(.7){\coev^L_X \, \otimes \, \coev_{\one,M}}
&& \mathtt{E}_{\cC} \otimes  \mathtt{E}_{\cM}
\ar[rr]^{\phi \, \otimes \, \id}
\ar@/^1pc/[rrdddd]^{\pi_{\cC}(X) \, \otimes \, \pi_{\cM}(M)}
&& \mathtt{E}_{\cD} \otimes  \mathtt{E}_{\cM}
\ar[dddd]^{\pi_{\cD}(X) \, \otimes \, \pi_{\cM}(M)}\\
\mathtt{E}_{\cC} \otimes  \mathtt{E}_{\cM}
\ar[ddd]_{\phi \, \otimes \, \id}
\ar@/_1.6pc/[rd]^(.25){\pi_{\cC}(X) \, \otimes \, \pi_{\cM}(M)}
&& &&\\
& X \otimes X^* \otimes \underline{\text{Hom}}(M,M)
\ar[r]^{\sim}
& X  \otimes \underline{\text{Hom}}(M,X^* \act M)
\ar[d]^{\id \, \otimes \, \underline{\text{Hom}}(M, e_{X^*,M})}
&&\\
&& X  \otimes \underline{\text{Hom}}(M,X^* \act M)
\ar[rrd]^{\sim}
&&\\
\mathtt{E}_{\cD} \otimes  \mathtt{E}_{\cM}
\ar[ruu]_{\; \pi_{\cD}(X) \, \otimes \, \pi_{\cM}(M)}
\ar[rrrr]^{\pi_{\cD}(X) \, \otimes \, \pi_{\cM}(M)}
&&&&  X \otimes X^* \otimes \underline{\text{Hom}}(M,M)
}
\]
}

\smallskip

\noindent The top left triangle commutes by Lemma~\ref{lem:EcHopfAlg}(a),~\eqref{eq:EcCc}, and  also by Lemma~\ref{lem:EmAlg}. The bottom left  and top right triangles commute by the universal property of $\phi$. The top middle region is Definition~\ref{def:omegaM}, and the bottom middle region commutes since $X^* \in \cD$ [Lemma~\ref{lem:Z2CM-tensor-subcat}].
\end{proof}

This brings us to the main result of this part.

\begin{theorem} \label{thm:wf-triv}
 If $\cM$ is weakly factorizable, then $\cM$ has trivial symmetric center. 
\end{theorem}

\begin{proof}
Assume $\cM$ is weakly factorizable, and denote $\cD:=\cZ_{2}(\cM)$. By Lemma~\ref{lem:Z2CM-tensor-subcat} and \cite[Lemma~5.7]{shimizu2019non}, it suffices to show that $\text{dim}_\kk \Hom_{\cD}(\one, \mathtt{C}_{\cD}) = 1$. Since $\Hom_{\cD}(\one, \mathtt{C}_{\cD}) \cong \Hom_{\cD}(\mathtt{E}_{\cD},\one)$ by \eqref{eq:EcCc}, and $\cD$ is a full subcategory of $\cC$, the result holds when:
\begin{equation} \label{eq:STS}
\text{dim}_\kk \Hom_{\cC}(\mathtt{E}_{\cD}, \one) = 1.
\end{equation}

Now retain the notation of Lemma~\ref{lem:Z2M-u}, and consider the map below: 
\[
\xymatrix@C=5pc@R=0.2pc{
 \beta: \Hom_{\cC}(\mathtt{E}_{\cD},\one)
 \ar[r]^{\Hom_{\cC}(\phi, \one)} 
 &\Hom_{\cC}(\mathtt{E}_{\cC},\one) \ar[r]^{\Omega_\cM} 
 &\Hom_{\cC}(\one,\mathtt{E}_{\cM}),\\ 
  f \ar@{|->}[r]
& f  \phi \ar@{|->}[r]
& (f \phi \otimes \id_{\mathtt{E}_{\cM}})\, \omega_{\cM}.
}
\]
We will show that $\beta$ is injective and $\text{dim}_\kk \, \text{im}(\beta) = 1$. This will establish \eqref{eq:STS}.

Lemma~\ref{lem:Z2M-u}(a) implies that $\phi$ is epic, so $\Hom_{\cC}(\phi, \one)$ is injective. Moreover, $\Omega_{\cM}$ is bijective by assumption. Thus, $\beta$ is injective. Next, Lemma~\ref{lem:Z2M-u}(c) implies that for $f\in \Hom_{\cC}(\mathtt{E}_{\cD},\one)$, 
\[ 
\beta(f) 
\; =  \; 
(f \phi \otimes \id_{\mathtt{E}_{\cM}})\, \omega_{\cM} 
\; =  \; 
(f \phi\otimes \id_{\mathtt{E}_{\cM}}) (u_{\mathtt{E}_{\cC}} \otimes u_{\mathtt{E}_{\cM}})
\; =  \; 
c(f) \, u_{\mathtt{E}_{\cM}}, 
\]
where $c(f)\in\kk$ is the scalar corresponding to the following endomorphism of $\one$:
\[
  c(f): \, \one \xrightarrow{u_{\mathtt{E}_{\cC}}} \mathtt{E}_{\cC} \xrightarrow{ \phi} \mathtt{E}_{\cD} \xrightarrow{f} \one.
\]
So, $\text{dim}_\kk \, \text{im}(\beta) \leq 1$.
Note that $\mathtt{E}_{\cD}$ is a Hopf algebra with counit $\varepsilon_{\mathtt{E}_{\cD}} \in  \Hom_{\cD}(\mathtt{E}_{\cD},\one) = \Hom_{\cC}(\mathtt{E}_{\cD},\one)$; see Lemma~\ref{lem:EcHopfAlg} and \eqref{eq:EcCc}.  Thus, by Lemma~\ref{lem:Z2M-u}(b), we get that
\[
c(\varepsilon_{\mathtt{E}_{\cD}}) 
\; = \; 
\varepsilon_{\mathtt{E}_{\cD}} \circ \phi \circ u_{\mathtt{E}_{\cC}} 
\; = \; 
\varepsilon_{\mathtt{E}_{\cD}} \circ u_{\mathtt{E}_{\cD}} 
\; = \; 
\id_{\one}.  
\]
Thus, $\beta(\mathtt{E}_{\cD})\neq 0$. Consequently,  $\text{dim}_\kk \, \text{im}(\beta) = 1$, which completes the proof.
\end{proof}

Naturally, Shimizu's Theorem [Theorem~\ref{thm:Shimizu}] prompts the inquiry below.

\begin{question} \label{ques:conv}
When do the converse statements of Lemma~\ref{lem:non-wf} and Theorem~\ref{thm:wf-triv} hold? 
\end{question}

Namely, Shimizu's arguments in the finite tensor case \cite{shimizu2019non} proceeds as follows: nondegeneracy $\Rightarrow$ weak-factorizability (cf. Lemma~\ref{lem:non-wf}); weak-factorizability $\Rightarrow$ trivial symmetric center (cf.  Theorem~\ref{thm:wf-triv}); trivial symmetric center $\Leftrightarrow$ factorizability; and factorizability $\Leftrightarrow$ nondegeneracy (cf. 
 Theorem~\ref{thm:FactNondeg}).
The difficulty of Question~\ref{ques:conv} lies in getting the trivial symmetric center condition to imply factorizability/ nondegeneracy; Definition~\ref{def:trivM} may need to be modified in future work.


\section{Factorizable comodule algebras} 
In this section, we characterize the nondegeneracy condition
of braided module categories over representation categories of quasitriangular Hopf algebras, by introducing the notion of factorizability for comodule algebras. We first provide background material on (quasitriangular) Hopf algebras  and (quasitriangular) comodule algebras in $\S$\ref{subsec:Hopf-comodules} and $\S$\ref{subsec:quasitriangular}. Our main result on factorizable comodule algebras is then presented in $\S$\ref{subsec:nondegenerate}.
Finally, in $\S$\ref{subsec:examples-ndcomodules}, we provide examples of factorizable comodule algebras, including those pertaining to the reflective centers from~$\S$\ref{sec:refcenter}.

\begin{notation} \label{not:Sec6}
In this section, we will use the following conventions:
\begin{enumerate}[(a)]
  \item For a finite-dimensional vector space $X$, we will denote its basis by $\{x_i\}$, and its dual basis by $\{x^i\}$, where $\langle x^i, x_j\rangle = \delta_{i,j}$ (Kronecker delta).
  \smallskip
  \item We will also use the Einstein convention to suppress summation symbols in expressions involving repeated indices.
\end{enumerate}
\end{notation}

\subsection{Hopf algebras and comodule algebras}\label{subsec:Hopf-comodules}

Take a Hopf algebra $H:=(H,m,u,\Delta,\varepsilon,S)$ over~$\Bbbk$. We will use the Sweedler notation for comultiplication, namely $\Delta(h) =: h_{(1)} \otimes_\Bbbk h_{(2)}$. Moreover, denote $\otimes:= \otimes_\Bbbk$. 
Then, the category 
\[
\cC:=H\text{-}\mathsf{FdMod}
\]
of finite-dimensional $\Bbbk$-modules over a finite-dimensional Hopf algebra $H$ is a finite tensor category.

\subsubsection{Preliminary Hopf identities}
When $H$ is finite-dimensional, the dual space $H^*$ is a Hopf algebra where, for all $f,f' \in H^*$ and $h, h' \in H$:
\begin{equation} \label{eq:dualHopf}
\begin{array}{c}
\langle ff', h \rangle := \langle f, h_{(1)} \rangle \; \langle f', h_{(2)} \rangle, 
\quad \quad  
\langle f, hh' \rangle := \langle f_{(1)}, h \rangle \; \langle f_{(2)}, h' \rangle,\\[.4pc]
\langle 1_{H^*}, h \rangle := \varepsilon_H(h),
\quad \quad 
\varepsilon_{H^*}(f):=\langle f, 1_H \rangle,
\quad \quad 
\langle S_{H^*}(f), h \rangle:=\langle f, S_H(h) \rangle.
\end{array}
\end{equation}

Moreover, for a basis $\{h_i\}$ of $H$, for all $f \in H^*$ and $h \in H$, we get that
\begin{equation} \label{eq:dualbasis}
\langle f, h_i \rangle \; h^i = f, \quad \quad 
\langle h^i, h \rangle \;  h_i = h.
\end{equation}

Note the following two straightforward identities for $X \in \cC$ with basis $\{x_i\}$, for $H$ with basis $\{h_i\}$, and for all $h, \, h', \, h'' \in H$:
\begin{align}
  \label{eq:xi-xi} (h\cdot x_i) \otimes \langle x^i,- \rangle & \; =  \; x_i \otimes \langle x^i, h \cdot (-) \rangle,  \\
  \label{eq:hi-hi} h' \, h_i \, h'' \otimes \langle h^i,- \rangle & \; =  \; h_i\otimes \langle h^i, h' (-) h''\rangle. 
\end{align}

Next, we recall some features of $\cC$ that will be needed later.

\begin{lemma}\label{lem:H-Mod}
Take $h\in H$. Then, for $X,\, Y\in \cC$, we have the statements below.
\begin{enumerate}[\upshape (a)]
\item $X\otimes Y\in\cC$ via $h\cdot(x\otimes y) = (h_{(1)}\cdot x) \otimes (h_{(2)}\cdot y)$, for $x\in X, \, y\in Y$.

\smallskip

\item The unit object of $\cC$ is $\kk_{\varepsilon}=\kk$ as a vector space with $H$-action given by $h\cdot 1 = \varepsilon(h)$.

\smallskip

\item The dual objects of $X \in \cC$ are $X^* = {}^*X =\Hom_\Bbbk(X,\Bbbk)$ as vector spaces with $H$-action given by $ \langle h\cdot x^*,-\rangle = \langle x^*, S(h) \cdot -\rangle$ and $ \langle h\cdot {}^*x,-\rangle = \langle {}^*x, S^{-1}(h) \cdot -\rangle$ for $x^*\in X^*$ and ${}^*x\in {}^*X$. 

\smallskip

\item The evaluation and coevaluation morphisms are given by: $\ev: X^*\otimes X\to \Bbbk$, $x^*\otimes x\mapsto \langle x^*,x\rangle$ and $\coev: \Bbbk\to X\otimes X^*$, $1\mapsto x_i\otimes x^i$. \qed
\end{enumerate}  
\end{lemma}

\subsubsection{Comodule algebras} See \cite{andruskiewitsch2007modcat} for more details here. A {\it left $H$-comodule algebra} is an algebra~$B$ over $\kk$ equipped with an algebra map 
\[
\delta : B \to H \otimes B, \quad b \mapsto b_{[-1]} \otimes b_{[0]} 
\]
making $B$ a left $H$-comodule. Here, the category 
\[
\cM:=B\text{-}\mathsf{FdMod}
\]
is a left $\cC$-module category via 
\[
\act: H\text{-}\mathsf{FdMod} \times B\text{-}\mathsf{FdMod} \to B\text{-}\mathsf{FdMod}, \quad \left((X, \cdot), (M, \ast)\right) \mapsto (X \otimes M,\;  \widetilde{\ast}),
\]
where $b \; \widetilde{\ast} \; (x \otimes m) := (b_{[-1]} \cdot x) \otimes (b_{[0]} \ast m)$ for $b \in B$, $x \in X$, $m \in M$. 

\medskip

For example, {\it left coideal subalgebras} of $H$ (which are left coideals of $H$ that are also subalgebras) are  left $H$-comodule algebras. 
Next, consider the  terms below for left $H$-comodule algebras $B$. 
\begin{itemize}[\footnotesize$\bullet$]
\item $B$ is {\it exact} if $\Bmod$ is an exact left $(\Hmod)$-module category. 

\smallskip

\item An {\it $H$-costable ideal} of $(B,\delta)$ is an ideal $I$ of $B$ such that $\delta(I) \subseteq H \otimes I$. 

\smallskip

\item $B$ is {\it $H$-simple} if it has no non-trivial $H$-costable ideal.

\smallskip

\item $B$ is {\it $H$-indecomposable} if there are no non-trivial 
$H$-costable ideals $I$ and $J$ where $B = I \oplus J$.
\end{itemize}
Observe that $H$-simplicity implies $H$-indecomposability.

\medskip

Next, we recall some features of $\cM = \Bmod$ that will be needed later.

\begin{lemma}
\cite[Propositions~1.18,~1.19,~1.20]{andruskiewitsch2007modcat} \cite[Theorem~6.1]{skyrabin2007proj} \label{lem:AM07}
Any indecomposable exact module category over $\Hmod$ is equivalent to $\Bmod$ as module categories, for some  $H$-indecomposable, exact $H$-comodule algebra $B$. Moreover, 
we have the following statements for the left $(\Hmod)$-module category $\Bmod$.
\begin{enumerate}[\upshape (a)]
\item If $B$ is a left coideal subalgebra of $H$, then $B$ is $H$-simple. 

\smallskip

\item If $B$ is an $H$-simple $H$-comodule algebra, then $\Bmod$ is exact.

\smallskip

\item If $B$ is an $H$-indecomposable $H$-comodule algebra, then $\Bmod$ is indecomposable.
\qed
\end{enumerate}
\end{lemma}

\begin{lemma}\cite[$\S$4.4]{shimizu2023relative}\label{lem:shirel-iHom}
Take modules $M,\, N\ \in \Bmod =: \cM$, with  $m\in M, \, n\in N$. Then,   the statements below hold for the module $X\in\Hmod =: \cC$.
\begin{enumerate}[\upshape (a)]
\item We have that $\iHom(M,N) = \Hom_{B}(H\act M,N)$ as a vector space.

\smallskip
    
\item The Hopf algebra $H$ acts on $\iHom(M,N)$ as $(h\cdot \xi)(h'\otimes m) = \xi(h'h\otimes m)$, for $\xi\in\iHom(M,N)$.

\smallskip
  
\item The map $\coev_{X,M}: X\rightarrow \iHom(M,X\act M)$ is given by 
$x \mapsto [ h\otimes m  \mapsto (h\cdot x)\otimes m ]$.

\smallskip
  
\item The map $\ev_{M,N}: \iHom(M,N)\otimes M\rightarrow N$ is given by
$\xi \otimes m \mapsto \xi(1_H\otimes m)$.

\smallskip
  
\item The isomorphism $X\otimes \iHom(M,N)\xrightarrow{\sim} \iHom(M,X\act N)$ is given by
\[ x\otimes \xi \mapsto [h\otimes m \mapsto (h_{(1)}\cdot x) \otimes \xi(h_{(2)}\otimes m)]. \]

\vspace{-.3in}

\qed

\end{enumerate}
\end{lemma}

Next, we recall the description of the end object $\tEM = \int_{M\in\cM} \iHom(M,M)$ of the category $\cC$. 

\begin{proposition}\cite[$\S$4.2] {bortolussi2021character}\label{prop:E_M-Hopf}
The following statements hold for the left module category \linebreak $\cM = \Bmod$ over $\cC = \Hmod$.
\begin{enumerate}[\upshape (a)]
\item The end object $\tE_{\cM}$ is a subspace of $\Hom_\kk(H,B)$ equal to:
\[ E(H,B) =\{ \xi:H\rightarrow B ~|~ \xi(b_{[-1]}h) b_{[0]} = b \xi (h), \; \forall b\in B,\, h\in H \}. \]
\item The Hopf algebra $H$ acts on $\tEM$ via $(h\cdot \xi)(h') = \xi(h'h)$ for all $h,h'\in H$.

\smallskip

\item The  maps $\pi_{\cM}(M):\tEM \rightarrow \iHom(M,M)$ are given by 
$\xi \mapsto [h\otimes m \mapsto \xi(h)\cdot m] $.
\qed
\end{enumerate}
\end{proposition}

In the special case when $B=H$ and $\delta=\Delta$, we obtain the following consequence.

\begin{corollary}\label{cor:E_C-Hopf} 
We have the statements below for $\cM = \Hmod = \cC$.
\begin{enumerate}[\upshape (a)]
  \item The end object $\tE_{\cC}=E(H,H)$ is isomorphic $H$ via the map $\xi \mapsto \xi(1_H)$.

  \smallskip
  
  \item The Hopf algebra $H$ acts on $\tEC=H$ acts on via $h\cdot h' = h_{(1)}h'S(h_{(2)})$, for $h,h'\in H$.
  
  \smallskip
  
  \item The  maps $\pi_{\cC}(X):\tEC\rightarrow X\otimes X^*$ are given by $\pi_{\cC}(X)(h) = (h\cdot x_i) \otimes x^i$, for $h\in H$. 
\end{enumerate}
\end{corollary}

\begin{proof}
The details can be derived from \cite[$\S$4.2] {bortolussi2021character}, especially \cite[Example~4.13]{bortolussi2021character}, but we provide some details about on part (a) for the reader's convenience.

The inverse of the map $\phi: E(H,H) \to H$, $\xi \mapsto \xi(1_H)$ is the map $\psi: H \to E(H,H)$ given by $\ell \mapsto  [h \mapsto h_{(1)} \ell S(h_{(2)})]$, for $\ell, h \in H$. Indeed, $\psi(\ell) \in E(H,H)$ since for all $h,h', \ell \in H$ we get that:

\vspace{-.1in}

{\small
\begin{align*}
[\psi(\ell)\bigl(h'_{(1)}h\bigr)]h'_{(2)} 
\; &= \; \bigl[\bigl(h'_{(1)}h\bigr)_{(1)}   \ell  S\bigl((h'_{(1)}h)_{(2)}\bigr)\bigr]h'_{(2)}\\[.2pc] 
\; &= \; h'_{(1)} h_{(1)}   \ell S(h_{(2)}) S(h'_{(2)}) h'_{(3)} 
= \; h'_{(1)} h_{(1)}   \ell S(h_{(2)}) \varepsilon(h'_{(2)})
\; = \; h' [\psi(\ell)(h)].
\end{align*}
}

\vspace{.05in}

\noindent Also, for all $\ell \in H$, we get that $\phi \psi(\ell) = \phi\bigl(h \mapsto h_{(1)} \ell S(h_{(2)})\bigr) = 1_{(1)} \ell S(1_{(2)}) = \ell$. So, $\phi \psi = \id_H$. On the other hand, for all $f \in E(H,H)$ and $h \in H$, we get that:

\vspace{-.1in}

{\small
\[
[\psi \phi(\xi)](h) \; = \;  h_{(1)} \xi(1_H) S(h_{(2)})
\overset{\xi \in E(H,H)}{=} 
 \xi\bigl(h_{(1)} 1_H\bigr) h_{(2)} S(h_{(3)})
 \; = \;   \xi\bigl(h_{(1)} 1_H\bigr) \varepsilon(h_{(2)})
  \; = \;   \xi(h).
\]
}

\vspace{.05in}

\noindent  So, $\psi \phi = \id_{E(H,H)}$.
\end{proof}


\subsection{Quasitriangular Hopf algebras and comodule algebras} \label{subsec:quasitriangular}
Retain Notation~\ref{not:Sec6}.

\subsubsection{Quasitriangular Hopf algebras}
See \cite[$\S$12.2]{radford2012hopf} for more details on the material here. We say that a Hopf algebra $H$ is {\it quasitriangular} if there exists an invertible element 
\[
R = R_i \otimes R^i \; \in H \otimes H \; \qquad \; \text{({\it $R$-matrix})}
\]
 satisfying the following properties:
 \[
\text{(i)} \; \; (\Delta \otimes \id_H) (R) = R_{13} R_{23},  
\quad \quad
\text{(ii)} \; \; (\id_H \otimes \Delta) (R) = R_{13} R_{12},
\quad \quad
\text{(iii)} \;  \; R \Delta(h) = \Delta^{\mathrm{op}}(h) R,  
\]
for  all $h \in H$. Here, $R_{cd}:=  R_i \; \text{($c$-th slot)} \otimes  R^i \; \text{($d$-th slot)}
\otimes  1_H \; \text{(other slots)}$,
e.g., $R_{13} =  R_i \otimes 1_H \otimes R^i$. 

\medskip

For an arbitrary invertible element $R:=R_i \otimes R^i \in H \otimes H$, we have that the natural isomorphism
\[
c_{X,Y} : X \otimes Y \to Y \otimes X, \quad
x \otimes y \mapsto  (R^i \cdot y) \otimes (R_i \cdot x)
\]
makes the category $\cC=\Hmod$ braided if and only if $R$ is an $R$-matrix for $H$.

\subsubsection{Factorizable Hopf algebras} \label{subsec:fact-Hopf} See \cite[Definition~2.1]{reshetikhin1988quantum} and \cite[$\S$12.4]{radford2012hopf}  for more details on the material here. We say that a  quasitriangular Hopf algebra $(H,R)$ is {\it factorizable} if 
\[
\theta_{(H,R)}: H^* \to H, \quad f \mapsto (f \otimes \id_H)(R^i R_j \otimes R_i R^j) = f(R^i R_j) R_i R^j
\]
is a vector space isomorphism. Here, $\theta_{(H,R)}$ is called the {\it Drinfeld map} of $(H,R)$. 

\smallskip

Note that the braided finite tensor category $\cC = H$-$\mathsf{FdMod}$ is nondegenerate if and only if $(H,R)$ is factorizable \cite[$\S$2.5]{shimizu2019recent}; see also \cite[$\S$7.4.6]{kerler2001non}.

\subsubsection{Quasitriangular comodule algebras} 
\label{sec:qtcomodalg} See \cite{kolb2020braided,laugwitz2023reflective} for more details on the material here. Assume that $(H,R)$ is a quasitriangular Hopf algebra. Let $B$ be a left 
$H$-comodule algebra with coaction $\delta$.
We say that $B$ is  {\emph{quasitriangular}} if it is equipped with an invertible element, 
\[
K:= \textstyle K_i \otimes K^i \; \in H \otimes B \; \qquad \; \text{({\it $K$-matrix})},
\]
satisfying the following properties:
\begin{equation*} \label{eq:K-axioms} \; \;
\text{(i)}\;  \; (\Delta \otimes \id_B) K = K_{23}R_{21} K_{13} R_{21}^{-1}, \quad \; \text{(ii)} \; \; (\id_H \otimes \delta) K = R_{21} K_{13} R_{12}, \quad \; \text{(iii)}\;  \; K \delta(b) = \delta(b) K, 
\end{equation*}
for all $b \in B$. Here, $K_{cd}:= K_i \; \text{($c$-th slot)}\; \otimes \; K^i \; \text{($d$-th slot)}
\otimes \; 1_H \; \text{(other slots)}.$

\medskip

For an arbitrary invertible element $K:=K_i \otimes K^i \in H \otimes B$, we have that the natural isomorphism
\begin{equation} \label{eq:braid-K}
e_{X,M}: X \otimes M \to X \otimes M, \quad
x \otimes m \mapsto \textstyle (K_i \cdot x) \otimes (K^i \ast m)
\end{equation}
yields a braiding on the module category $B\text{-}\mathsf{FdMod}$  if and only if $K$ is a $K$-matrix for $B$.

\begin{example} \label{ex:overVec}
\begin{enumerate}[\upshape (a)]
\item  Take a triangular Hopf algebra $(H, R)$, that is, $R_{21}=R^{-1}$. Then, any left $H$-comodule algebra $B$  is quasitriangular with $K$-matrix $K = 1_H \otimes 1_B$.

\smallskip

\item Continuing Example~\ref{ex:M=C,double} for the identity braided tensor functor, we get that $(B,\delta) = (H,\Delta)$ is a quasitriangular left $H$-comodule algebra with $K(x \otimes y) = (R^i R_j \cdot y) \otimes (R_i R^j \cdot x)$, for $x \in X$, $y \in Y$ with $X,Y \in H$-$\mathsf{FdMod}$. That is, $B = H$ is quasitriangular with $K = R_{21} R$. 
\end{enumerate}
\end{example}

Next, we discuss a special case of Example~\ref{ex:M=C,double} in the Hopf setting, when the braided tensor functor that is not necessarily the identity.

\begin{example} \label{ex:kG-FdMod}
Take a finite group $G$, with subgroup $G'$, and consider the braided tensor functor $\Bbbk G$-$\mathsf{FdMod} \to \Bbbk G'$-$\mathsf{FdMod}$ given by restriction. The braidings for the categories here are given by the flip map. This corresponds to a quasitriangular left $\Bbbk G$-comodule algebra structure on $\Bbbk G'$, where $\delta = \Delta_{\Bbbk G'}$ and $K = 1_{\Bbbk G} \otimes 1_{\Bbbk G'}$. Moreover, $\Bbbk G'$ is an exact and indecomposable left $\Bbbk G$-comodule algebra; see, e.g., \cite[Corollary 7.12.20]{etingof2015tensor}.
\end{example}


\subsection{Factorizable comodule algebras}
\label{subsec:nondegenerate}
Consider the following hypothesis for the material below.

\begin{hypothesis} \label{hyp:fact} 
Let $(H,R)$ be a finite-dimensional quasitriangular Hopf algebra, and let $(B,K)$ be a finite-dimensional quasitriangular left $H$-comodule algebra. Also, assume that $B$-$\mathsf{FdMod}$ is an indecomposable and exact left module category over $H$-$\mathsf{FdMod}$. (Compare to Hypothesis~\ref{hyp:Sec4}.)
\end{hypothesis}

Note that Hypothesis~\ref{hyp:fact} holds when the $H$-comodule algebra $B$ is $H$-simple by Lemma~\ref{lem:AM07}(b,c). Next, we introduce the main property of $H$-comodule algebras of interest here. 

\begin{definition} \label{def:thetaB}
We call $(B,K)$ a {\it factorizable $H$-comodule algebra} if the following map is an isomorphism of vector spaces:
\[ \theta_B := \theta_{(B,K)}: H^* \rightarrow E(H,B), \quad   f\mapsto [h\mapsto \langle f, S(h_{(1)})K_i h_{(2)}\rangle K^i ]. \]
\end{definition}

This notion is well-defined to the result below.

\begin{lemma} \label{lem:image-thetaB}
The image of the map $\theta_B$ above is in $E(H,B)$.
\end{lemma}

\begin{proof}
The result holds due to the calculation below, for $f \in H^*$, $h \in H$, and $b \in B$:
{\small
\begin{align*}
\bigl[\theta_B(f)(b_{[-1]} h)\bigr] b_{[0]} 
& \; = \; \bigl\langle f, \; S\bigl((b_{[-1]} h)_{(1)}\bigr) K_i (b_{[-1]}h)_{(2)} \bigr\rangle \; K^i  \, b_{[0]}\\[.2pc]
& \; = \; \bigl\langle f, \; S(h_{(1)}) S(b_{[-2]}) K_i b_{[-1]} h_{(2)} \bigr\rangle \;  K^i  \, b_{[0]}\\
& \; \overset{\S\textnormal{\ref{sec:qtcomodalg}}\text{(iii)}}{=} \; \bigl\langle f, \; S(h_{(1)}) S(b_{[-2]}) b_{[-1]} K_i  h_{(2)} \bigr\rangle \;  b_{[0]} \, K^i \\[.2pc]
& \; = \; \bigl\langle f, \; S(h_{(1)}) \varepsilon(b_{[-1]}) K_i  h_{(2)} \bigr\rangle  \; b_{[0]} \, K^i \\[.2pc]
& \; = \; \bigl\langle f, \; S(h_{(1)})  K_i  h_{(2)} \bigr\rangle \;  b  \, K^i \; = \; b \;  \bigl[\theta_B(f)(h)\bigr]. 
\end{align*}
}

\vspace{-.25in}

\end{proof}

 Next, we have a useful characterization of factorizability for comodule algebras.

\begin{lemma} \label{lem:theta}
We have that $\textnormal{dim}_\kk E(H,B) = \textnormal{dim}_\kk H$. As a consequence, $B$ is factorizable if and only if the linear map $\theta_B$ in Definition~\ref{def:thetaB} is injective. 
\end{lemma}

\begin{proof}
We employ Frobenius-Perron dimension here; see \cite{etingof2015tensor} for details. To start, note that $\text{dim}_\kk E(H,B) =  \text{FPdim}_{\mathsf{FdVec}} (E(H,B))$. Next, $\text{FPdim}_{\mathsf{FdVec}} (E(H,B)) = \text{FPdim}_{\cC} (E(H,B))$, for $\cC:= \Hmod$,  by \cite[Proposition~4.5.7]{etingof2015tensor} applied to the forgetful functor $\cC \to \mathsf{FdVec}$. By \cite[Theorem~4.8]{bortolussi2021character}, we get that $\text{FPdim}_{\cC} (E(H,B)) = \text{FPdim}_{\cC} (\tEM)$, for $\cM:= \Bmod$. Now apply \cite[Lemma~4.3(iv)]{bortolussi2021character} to get that $\text{FPdim}_{\cC} (\tEM) = \text{FPdim}(\cC)$, which is equal to $\text{dim}_\kk H$ (see, e.g., \cite[Example~6.1.9]{etingof2015tensor}). This proves the first statement. The last statement is clear.
\end{proof}

\begin{example} \label{ex:fact-reg}
\begin{enumerate}[\upshape (a)]
\item Take $(H,R)$ triangular as in Example~\ref{ex:overVec}(a). Then, for a  quasitriangular left $H$-comodule algebra $(B,K= 1_H \otimes 1_B)$, we have that $[\theta_B(f)](h) \overset{\eqref{eq:dualHopf}}{=} \varepsilon_H(h) \, \varepsilon_{H^*}(f) \, 1_B$, for $h \in H$, $f \in H^*$. Now $(B,K= 1_H \otimes 1_B)$ is factorizable if and only if $\theta_B$ is injective [Lemma~\ref{lem:theta}], and this happens precisely when  $H = \kk$.

\medskip

\item Take $B=H$ with $K=R_{21}R$ as in  Example~\ref{ex:overVec}(b). Here, $B$ is a left coideal subalgebra of $H$, so it is $H$-simple; hence, Hypothesis~\ref{hyp:fact} holds [Lemma~\ref{lem:AM07}]. Using the isomorphism $E(H,H)\cong H$ from Corollary~\ref{cor:E_C-Hopf}(a), the map $\theta_{(H, R_{21}R)}$ is given by
\[
\theta_{(H, R_{21}R)}: H^* \rightarrow H, \quad f \mapsto \langle f, S(1_{(1)})K_i 1_{(2)}\rangle K^i = \langle f,K_i\rangle K^i = \langle f, R^i R_j \rangle R_i R^j.
\]
This is equal to the map $\theta_{(H,R)}$ from $\S$\ref{subsec:fact-Hopf}. So, the quasitriangular left $H$-comodule algebra $(H, R_{21}R)$ is factorizable if and only if the quasitriangular Hopf algebra $(H,R)$ is factorizable.
\end{enumerate}
\end{example}

Now, the main result of this section is the following.

\begin{theorem}\label{thm:nondeg-comodules}
  The braided left $(\Hmod)$-module category $\Bmod$ is nondegenerate if and only if $(B,K)$ is a factorizable left $H$-comodule algebra.
\end{theorem}

In order to prove this result, we need an explicit description of the copairing $\omega_{\cM}: \one \rightarrow \tEC\otimes \tEM$  from Definition~\ref{def:omegaM}, for $\cC = \Hmod$ and $\cM = \Bmod$ here. 

\begin{proposition} \label{prop:omegaB}
Recall Lemma~\ref{lem:H-Mod}(b), Proposition~\ref{prop:E_M-Hopf}(a) and Corollary~\ref{cor:E_C-Hopf}(b).   Then, the copairing  $\omega_B : \kk_{\varepsilon} \rightarrow H \otimes E(H,B)$ given by 
  \[ 1 \mapsto h_i \otimes \big[h \mapsto \bigl< h^i, S\left( S(h_{(1)})K_j h_{(2)}  \right) \bigr> K^j \big] 
  \]
  is the copairing $\omega_{B\text{-}\mathsf{FdMod}}$ from Definition~\ref{def:omegaM}.
\end{proposition}

\begin{proof}
The image of the map $\omega_B$ lands in $H \otimes E(H,B)$ with an argument similar to the proof of Lemma~\ref{lem:image-thetaB}. Now it remains to show that, first, $\omega_B$ is a morphism in $\Hmod$, and second, that with our choice of $\omega_B$, the diagram in Definition~\ref{def:omegaM} commutes. Then, as there is a unique morphism in $\Hmod$ from $\kk_{\varepsilon}$ to $H\otimes E(H,B)$ such that diagram commutes, we will be done.

We start by checking that $\omega_B$ is a morphism in $\Hmod$. We need to show that:
\[ \omega_B(h\cdot 1) = h\cdot \omega_B(1). \]
The left-hand-side is $\varepsilon(h)\omega_B(1)$ [Lemma~\ref{lem:H-Mod}(b)]. The right-hand-side is:
{\small
\begin{align*}
  h\cdot\omega_B(1)
  & = \; h\cdot \bigl(h_i \otimes \big[h' \mapsto \bigl< h^i, S\left( S(h'_{(1)})K_j h'_{(2)}  \right) \bigr> K^j \big] \bigr) \\
  & \overset{\text{Lem.~\ref{lem:H-Mod}(a)}}{=} \;  \left(h_{(1)}\cdot h_i\right) \otimes \left(h_{(2)}\cdot \big[h' \mapsto \bigl< h^i, S\left( S(h'_{(1)})K_j h'_{(2)}  \right) \bigr> K^j \big]\right) \\
  & \overset{\text{Cor.~\ref{cor:E_C-Hopf}(b)}}{=}  \; h_{(1)}h_i S(h_{(2)}) \otimes \left( h_{(3)}\cdot \big[h' \mapsto \bigl< h^i, S\left( S(h'_{(1)})K_j h'_{(2)}  \right) \bigr> K^j \big] \right) \\
  & \overset{\text{Prop.~\ref{prop:E_M-Hopf}(b)}}{=} \; h_{(1)}h_i S(h_{(2)}) \otimes \big[h' \mapsto \bigl< h^i, S\bigl( S((h'h_{(3)})_{(1)})K_j (h'h_{(3)})_{(2)}  \bigr) \bigr> K^j \big] \\
  & = \; h_{(1)}h_i S(h_{(2)}) \otimes \big[h' \mapsto \bigl< h^i, S\bigl( S(h'_{(1)}h_{(3)}) K_j h'_{(2)} h_{(4)} \bigr) \bigr> K^j \big] \\
  & \overset{{\eqref{eq:hi-hi}}}{=} \; h_i  \otimes \big[h' \mapsto \bigl< h^i, h_{(1)} S\bigl( S(h'_{(1)}h_{(3)}) K_j h'_{(2)} h_{(4)} \bigr) S(h_{(2)}) \bigr> K^j \big]  \\[.2pc]
  & = \;  h_i  \otimes \big[h' \mapsto \bigl< h^i, h_{(1)} S\bigl( S(h_{(3)}) S(h'_{(1)}) K_j h'_{(2)} h_{(4)} \bigr) S(h_{(2)}) \bigr> K^j \big] \\[.2pc]
  & = \; h_i  \otimes \big[h' \mapsto \bigl< h^i, h_{(1)} S(h_{(4)}) S\bigl(S(h'_{(1)}) K_j h'_{(2)}\bigr)  S^2(h_{(3)}) S(h_{(2)}) \bigr> K^j \big] \\[.2pc]
  & = \; h_i  \otimes \big[h' \mapsto \bigl< h^i, h_{(1)} S(h_{(3)}) S\bigl(S(h'_{(1)}) K_j h'_{(2)}\bigr) \varepsilon(h_{(2)}) \bigr> K^j \big] \\[.2pc]
  & = \; h_i  \otimes \big[h' \mapsto \bigl< h^i, h_{(1)} S(h_{(2)}) S\bigl(S(h'_{(1)}) K_j h'_{(2)}\bigr)   \bigr> K^j \big] \\[.2pc]
  & = \; h_i  \otimes \big[h' \mapsto \bigl< h^i, \varepsilon(h) S\bigl(S(h'_{(1)}) K_j h'_{(2)}\bigr)   \bigr> K^j \big] \\[.2pc]
  & = \;  \varepsilon(h) \; \omega_B(1).
\end{align*}
}

\noindent Thus, $\omega_B$ is a morphism in $\Hmod$. 

Next, to prove that $\omega_B = \omega_{B\text{-}\mathsf{FdMod}}$, we need to check that the following diagram commutes:
{\small
\[
\hspace{-.02in}
\xymatrix@C=2pc@R=1pc{ 
\kk_{\varepsilon}
\ar[r]^(.45){\omega_B}
\ar[d]_(.4){\coev_X\; \otimes \; \textnormal{coev}_{\kk_{\varepsilon}, M}}
& 
\mathtt{E}_{\cC} \otimes \mathtt{E}_{\cM}
\ar[r]^(.38){\pi_{\cC}(X)\;  \otimes \; \pi_{\cM}(M)}
& 
X \otimes X^* \otimes \; \Hom_B(H\act M,M)
\ar[d]_{\sim}^{\eqref{eq:internalHom-2}}
\\
X \otimes X^* \otimes \; \Hom_B(H\act M,M)
\ar@<-2pt>[r]^{\sim}_{\eqref{eq:internalHom-2}}
&
X  \otimes \; \Hom_B(H\act M,X^*\act M)
\ar@<-2pt>[r]^{\sim}_{\id \; \otimes \; \iHom(\id, e_{X^{\scalebox{.6}{$*$}},M})}
& 
X  \otimes \; \Hom_B(H\act M,X^*\act M).
}
\]
}
Going along the top-right route, we get:
{\small
\begin{align*}
  1\;\; 
  \xmapsto{\qquad\;\;\;\; \omega_B \qquad \;\;\;\;}
  & \; h_i \otimes \big[h \mapsto \bigl< h^i, S\left( S(h_{(1)})K_j h_{(2)}  \right) \bigr> K^j \big] 
  \\
  {}^{[\text{Cor.~\ref{cor:E_C-Hopf}(c)}]} 
  \xmapsto{\qquad\pi_{\cC}(X) \; \otimes \; \id_{\mathtt{E}_{\cM}} \qquad}
  & \; (h_i\cdot x_k) \otimes x^k \otimes \big[h \mapsto \bigl< h^i, S\left( S(h_{(1)})K_j h_{(2)}  \right) \bigr> K^j \big] 
  \\
  {}^{\text{[Prop.~\ref{prop:E_M-Hopf}(c)]}} 
  \xmapsto{\id_{X\otimes X^* }\;  \otimes \;  \pi_{\cM}(M)}
  & \; (h_i\cdot x_k) \otimes x^k  \otimes \big[h\otimes m \mapsto \bigl< h^i, S\left( S(h_{(1)})K_j h_{(2)}  \right) \bigr> (K^j\cdot m) \big] 
  \\
  {}^{\text{[Lem.~\ref{lem:shirel-iHom}(e)]}} 
  \xmapsto{\; \qquad \eqref{eq:internalHom-2} \qquad\;}
  & \; (h_i\cdot x_k) \otimes \big[h\otimes m \mapsto (h_{(1)}\cdot x^k) \otimes \bigl< h^i, S\left( S(h_{(2)})K_j h_{(3)}  \right) \bigr> (K^j\cdot m) \big] .
\end{align*}
}

\noindent Note that the output above is equal to:
{\small
\begin{align*}
  & (h_i\cdot x_k) \otimes \big[h\otimes m \mapsto (h_{(1)}\cdot x^k) \otimes \bigl< h^i, S\left( S(h_{(2)})K_j h_{(3)}  \right) \bigr> (K^j\cdot m) \big] 
  \\
  &\overset{\eqref{eq:xi-xi}}{=} \;
    x_k \otimes \big[h\otimes m \mapsto h_{(1)}\cdot\langle x^k, h_i \cdot (-)\rangle \otimes \bigl< h^i, S\left( S(h_{(2)})K_j h_{(3)}  \right) \bigr> (K^j\cdot m) \big]
  \\
  &\overset{\text{Lem.~\ref{lem:H-Mod}(c)}}{=}
    x_k \otimes \big[h\otimes m \mapsto \langle x^k, h_i S(h_{(1)}) \cdot (-) \rangle \otimes \bigl< h^i, S\left( S(h_{(2)})K_j h_{(3)}  \right) \bigr> (K^j\cdot m) \big]
  \\
  &\overset{\eqref{eq:dualbasis}}{=}\;
   x_k \otimes \big[h\otimes m \mapsto \langle x^k, S\left( S(h_{(2)})K_j h_{(3)} \right) S(h_{(1)}) \cdot (-) \rangle \otimes (K^j\cdot m) \big]
  \\[.2pc]
  &= \;
   x_k \otimes \big[h\otimes m \mapsto \langle x^k, S\left(h_{(1)}S(h_{(2)}) K_j h_{(3)} \right) \cdot (-) \rangle \otimes (K^j\cdot m) \big]
  \\[.2pc]
  &= \;
   x_k \otimes \big[h\otimes m \mapsto \langle x^k, S\left( \varepsilon(h_{(1)}) K_j h_{(2)} \right) \cdot (-) \rangle \otimes (K^j\cdot m) \big]
  \\[.2pc]
  &= \;
   x_k \otimes \big[h\otimes m \mapsto \langle x^k, S\left(  K_j h \right) \cdot (-) \rangle \otimes (K^j\cdot m) \big]
  \\
  & \overset{\text{Lem.~\ref{lem:H-Mod}(c)}}{=} \;
   x_k \otimes \big[h\otimes m \mapsto (K_j h \cdot x^k) \otimes (K^j\cdot m) \big].
\end{align*}
}

Next, we calculate the result obtained by going along the left-bottom route of the diagram, starting by using Lemmas~\ref{lem:H-Mod}(d) and~\ref{lem:shirel-iHom}(c):
{\small
\begin{align*}
  1
  \xmapsto{\coev_X \; \otimes  \; \coev_{\kk_{\varepsilon}, M}}
  & \; x_k \otimes x^k \otimes \big[ h\otimes m \mapsto (h\cdot 1)m =\varepsilon(h)m \big] 
  \\
   {}^{[\text{Lem.~\ref{lem:shirel-iHom}(e)}]}  \xmapsto{\qquad\, \eqref{eq:internalHom-2} \,\qquad}
  & \; x_k \otimes \big[ h\otimes m \mapsto (h_{(1)}\cdot x^k) \otimes \varepsilon(h_{(2)})m = (h\cdot x^k)\otimes m \big]
  \\
   {}^{[\text{Lem.~\ref{lem:shirel-iHom}(a)}, \eqref{eq:braid-K}]} \xmapsto{\id \, \otimes \, \iHom(\id, e_{X^{\scalebox{.6}{$*$}},M}) \;} 
  & \; x_k \otimes \big[ h\otimes m \mapsto (K_j\cdot(h\cdot x^k))\otimes (K^j\cdot m)  = (K_j h \cdot x^k) \otimes (K^j\cdot m)  \big].
\end{align*}
}

\noindent As the results obtained by going along the top-right route and the left-bottom route are the same, the diagram commutes. Hence, the proof is complete.
\end{proof}

With the description of $\omega_{\cM}$ in hand, we can now describe the map $\theta_{\cM}:\tEC^*\rightarrow\tEM$ from  Definition~\ref{def:Mnondeg}, for $\cC = \Hmod$ and $\cM = \Bmod$ here. 

\begin{proposition} \label{prop:thetaBmod}
 The map $\theta_{B\text{-}\mathsf{FdMod}}: H^* \rightarrow E(H,B)$ is given by
  \[ f \mapsto \bigl[ h \mapsto   \bigl< f, S
\left( S(h_{(1)}) K_i h_{(2)}\right) \bigr> K^j \bigr] .\]
\end{proposition}

\begin{proof}
 Recall from Corollary~\ref{cor:E_C-Hopf}(b) that $\tEC^* \cong H^*$, and from Proposition~\ref{prop:E_M-Hopf}(a) that $\tEM \cong E(H,B)$. Then, we compute, for all $f \in H^*$ that
 {\small
 \begin{align*}
 \theta_{B\text{-}\mathsf{FdMod}}(f) 
 & \; \overset{\text{Def.~\ref{def:Mnondeg}}}{=}\; 
 (\ev^L_H \otimes \id_{E(H,B)})(\id_{H^*} \otimes \omega_{B\text{-}\mathsf{FdMod}})(f) \\
  & \; \overset{\text{Prop.~\ref{prop:omegaB}}}{=}\; 
 (\ev^L_H \otimes \id_{E(H,B)})(f \otimes h_i \otimes \big[h \mapsto \bigl< h^i, S\left( S(h_{(1)})K_j h_{(2)}  \right) \bigr> K^j \big])\\[.2pc]
  & \; =
  \; 
 \langle f, h_i \rangle \big[h \mapsto \bigl< h^i, S\left( S(h_{(1)})K_j h_{(2)}  \right) \bigr> K^j \big])
 \\[.2pc]
  & \; \overset{\eqref{eq:dualbasis}}{=}
  \; 
  \big[h \mapsto \bigl<f, S\left( S(h_{(1)})K_j h_{(2)}  \right) \bigr> K^j \big]).
 \end{align*}
 }

 \vspace{-.25in}
 
\end{proof}

Now, we can prove the main result of this section.

\begin{proof}[Proof of Theorem~\ref{thm:nondeg-comodules}]
The antipode of $H^*$, denoted $S_{H^*}$, is given by $S_{H^*}(f) = f \circ S$. 
So, we obtain that $\theta_{B\text{-}\mathsf{FdMod}}$ from Proposition~\ref{prop:thetaBmod} is equal to $\theta_{B} \circ S_{H^*}$, where $\theta_{B}$ is given in Definition~\ref{def:thetaB}. Namely, we compute, for all $f \in H^*$:
{\small
\begin{align*}
\theta_{B} \circ S_{H^*}(f) \;  =  \; [h\mapsto \langle f \circ S, \; S(h_{(1)})K_i h_{(2)}\rangle K^i ]
 \; =  \;  [h\mapsto \langle f, \; S \left(S(h_{(1)})K_i h_{(2)} \right) \rangle K^i  ]
  \;=  \;
\theta_{B\text{-}\mathsf{FdMod}}(f).
\end{align*}
}
As $H$ is finite-dimensional, $S_{H^*}$ is bijective. Hence, $\theta_{{B\text{-}\mathsf{FdMod}}}$ is bijective  if and only if $\theta_B$ is bijective. Therefore,  $B\text{-}\mathsf{FdMod}$ is nondegenerate as a braided module category if and only if $B$ is factorizable as a quasitriangular comodule algebra.
\end{proof}

Now we illustrate Theorem~\ref{thm:nondeg-comodules} by continuing Example~\ref{ex:kG-FdMod}.

\begin{example} \label{ex:kG-FdMod-2}
Take $H = \Bbbk G$ and $B = \Bbbk G'$, for $G' \leq G$ as in Example~\ref{ex:kG-FdMod}. By using~\eqref{eq:dualHopf}, $\theta_{\Bbbk G'}$ from Definition~\ref{def:thetaB} sends $f \in (\Bbbk G)^*$ to $[g \mapsto \varepsilon_{(\Bbbk G)^*}(f) 1_{\kk G'}]$, for $g \in G$. Since $\varepsilon_{(\Bbbk G)^*}(\delta_g) = 0$ for all $g \in G$ not equal to $e$ (with $\{\delta_g\}_{g \in G}$ being the dual basis for $(\Bbbk G)^*$), $\theta_{\Bbbk G'}$ is injective if and only if $G$ is the trivial group $\langle e \rangle$. Thus, $\kk G'$ is a factorizable left $\kk G$-comodule algebra only when $G = \langle e \rangle$.

This is consistent with Example~\ref{ex:fact-reg}(b) in the case when $H = B = \kk G$ with $R = K = 1_{\kk G} \otimes 1_{\kk G}$. Namely, the following are equivalent: $\kk G$ is a factorizable left $\kk G$-comodule algebra; $\kk G$-$\mathsf{FdMod}$ is nondegenerate as a regular module category; $\kk G$-$\mathsf{FdMod}$ is nondegenerate as a braided finite tensor category; $\kk G$-$\mathsf{FdMod} \cong \mathsf{FdVec}$; and $G = \langle e \rangle$.
\end{example}


\subsection{Reflective algebras as factorizable comodule algebras}\label{subsec:examples-ndcomodules}
In this part, we present several families of factorizable comodules (in addition to those discussed in Example~\ref{ex:fact-reg}).

The following construction is from \cite[$\S\S$5,6]{laugwitz2023reflective}.
Let $(H, R = R_i \otimes R^i)$ be a finite-dimensional quasitriangular Hopf algebra  with dual basis $\{h_k, h^k\}_k$. Next, take $\widehat{H}$  to be the left $H$-module coalgebra that is equal to $H$ as vector spaces, and with the following comultiplication, counit, and left $H$-action formulas, for $h, \ell \in H$:
\[
\widehat{\Delta}(h) := \; 
\textstyle  R^j  h_{(1)}  R^i  \, \otimes \,  h_{(2)}  R_i  S^{-1}(R_j), \quad \quad
\widehat{\varepsilon}(h) := \; \varepsilon(h),\quad \quad
\ell \rightharpoonup h := \; \ell_{(2)}  h  S^{-1}(\ell_{(1)}).
\]
Therefore, $(\widehat{H}^*)^{\textnormal{op}}$ is a right $H^{\textnormal{cop}}$-module algebra via $\langle f \leftharpoonup \ell, h \rangle := \langle f, \ell \rightharpoonup h \rangle$,  for $h, \ell \in H$, and $f \in H^*$.
Now for a nonzero  left $H$-comodule algebra $A$, define the {\emph{reflective algebra of $A$ with respect to $H$}} to be the crossed product algebra:
\[
R_H(A):= A \rtimes_H (\widehat{H}^*)^{\textnormal{op}}   := A \circledast (\widehat{H}^*)^{\textnormal{op}}/ \langle f a - a_{[0]} (f \leftharpoonup a_{[-1]} ) \rangle_{a \in A, \; f \in (\widehat{H}^*)^{\textnormal{op}}}, 
\]
where $\circledast$ denotes the free product of algebras. These algebras represent reflective centers $\cE_{\cC}(\cM)$ in the Hopf case, which we see as follows. 

\begin{theorem}\cite[Theorem~6.6 and Corollary~6.8]{laugwitz2023reflective} \label{thm:LWYmain}
Retain the notation above. Then, 
\begin{equation*} \label{eq:EHA-intro}
\cE_{H\text{-}\mathsf{FdMod}}(A\text{-}\mathsf{FdMod}) \; \cong \; R_H(A)\text{-}\mathsf{FdMod}
\end{equation*}
as braided left module categories over $H\text{-}\mathsf{FdMod}$. In particular, the statements below hold.
\begin{enumerate}[\upshape (a)]
\item  $R_H(A)$ is a left $H$-comodule algebra with left $H$-coaction $\delta_{\textnormal{ref}}$ on $R_H(A)$ given by:
\[
\hspace{0.5in} \delta_{\textnormal{ref}}(a) := a_{[-1]} \otimes a_{[0]},\quad \quad
\delta_{\textnormal{ref}}(f) := \langle f,   R^j h_k R_i \rangle  R_j R^i  \otimes  h^k, \quad \quad
\delta_{\textnormal{ref}}(a  f):= \delta_{\textnormal{ref}}(a)\; 
\delta_{\textnormal{ref}}(f), 
\]
for $a \in A$ and $f \in (\widehat{H}^*)^{\textnormal{op}}$.

\medskip

\item $R_H(A)$ is quasitriangular as an $H$-comodule algebra with $K$-matrix:
\begin{equation*}
\label{eq:K-A}
K_{\textnormal{ref}}(A) := h_k \otimes h^k \; \in H \otimes (\widehat{H}^*)^{\textnormal{op}} \; \subset H \otimes R_H(A). 
\end{equation*}
Here, $K_{\textnormal{ref}}(A)$ does not depend on the choice of dual bases of $H$. \qed
\end{enumerate}
\end{theorem}

As discussed in \cite{laugwitz2023reflective}, reflective algebras are the module-theoretic version of Drinfeld doubles of finite-dimensional Hopf algebras; the latter of which are key examples of factorizable Hopf algebras. 
Finally, we present a sufficient condition for  reflective algebras to be factorizable comodule algebras.

\begin{proposition} \label{prop:RHA-fact}
If the reflective algebra $R_H(A)$ is $H$-simple, then $R_H(A)$ is a factorizable left $H$-comodule algebra. 
\end{proposition}

\begin{proof} 
The $H$-simplicity condition implies that $R_H(A)$ satisfies Hypothesis~\ref{hyp:fact}.
Now by Lemma~\ref{lem:theta}, it suffices to show that $\theta_{R_H(A)}: H^* \to E(H, R_H(A))$ from Definition~\ref{def:thetaB} is injective. Take $f,f' \in H^*$ and assume that $\theta_{R_H(A)}(f) = \theta_{R_H(A)}(f')$. Then, by Theorem~\ref{thm:LWYmain}(b), we get that 
\[
\langle f, S(h_{(1)}) h_k h_{(2)} \rangle h^k \; = \; \langle f', S(h_{(1)}) h_k h_{(2)} \rangle h^k,
\]
for all $h \in H$. Therefore,  
$f  = \langle f,  h_k  \rangle h^k  =  \langle f',  h_k  \rangle h^k = f'$ using $h = 1_H$, along with \eqref{eq:dualbasis}. Thus, $\theta_{R_H(A)}$ is injective, as desired.
\end{proof}

We illustrate the result above using the Drinfeld double of a finite group; see \cite[$\S$6.4]{laugwitz2023reflective}.

\begin{example}
Take a finite group $G$, and take  $\{x\}_{x \in G}$ and $\{\delta_x\}_{x \in G}$ to be the dual bases of $\Bbbk G$ and $(\Bbbk G)^*$, resp. Then, the Drinfeld double $D(G)$ of $\Bbbk G$ and its dual $D(G)^*$ have bases $\{\delta_x y\}_{x,y \in G}$ and $\{x \delta_y\}_{x,y \in G}$, resp. It is well known that $D(G)$ is a finite-dimensional quasitriangular Hopf algebra. By \cite[Proposition~6.18]{laugwitz2023reflective}, the reflective algebra $R_{D(G)}(\Bbbk)$ has basis $\{x \delta_y\}_{x,y \in G}$, with product,
\[
\big( x \delta_y \big) \big( x' \delta_{y'} \big)= 
\delta_{y', y^{-1} x^{-1} y x y} \; y^{-1} x y x' y^{-1} x^{-1} y x \; \delta_y,
\]
and left $D(G)$-comodule structure,
$
\textstyle \delta_{\textnormal{ref}} (x \delta_y) = 
\textstyle \sum_{g \in G} 
\delta_g \,  y^{-1} x y \, \otimes\,  g^{-1} x g \, \delta_{g^{-1} y},
$
and with $K$-matrix given by $K=\sum_{g,h \in G}  \delta_g h \otimes g \delta_h$.
We claim that $R_{D(G)}(\Bbbk)$ is $D(G)$-simple. To see this, take a nonzero $D(G)$-costable ideal $I$ of $R_{D(G)}(\Bbbk)$, with a nonzero element $f \in I$. Then, there is a  term  $\alpha \, x \delta_y$ of $f$, for $\alpha \in \Bbbk^\times$. By $D(G)$-costability,  $g^{-1} x g \; \delta_{g^{-1} y} \in I$ for each $g \in G$. Taking $g = y \ell^{-1} $ with $a:=\ell y^{-1} x y \ell^{-1}$, we get that $a \, \delta_\ell \in I$, for each $\ell \in G$. Next, $I$ is an ideal, so for each $h \in G$,
\[
h \, \delta_\ell
\; = \; 
(a \, \delta_\ell)
(\ell^{-1}a^{-1}\ell h a^{-1} \ell^{-1} a \ell \; \delta_{\ell^{-1} a^{-1} \ell a \ell}) 
\; \; \in I.
\]
Thus, $h \delta_\ell \in I$ for all $h, \ell \in G$, and hence, $I = R_{D(G)}(\Bbbk)$. So, $R_{D(G)}(\Bbbk)$ is $D(G)$-simple. Now by Proposition~\ref{prop:RHA-fact}, the reflective center $R_{D(G)}(\Bbbk)$ is a factorizable left $D(G)$-comodule algebra.
\end{example}

\section*{Acknowledgements} 

We are grateful to the anonymous referees who supplied careful and insightful comments that improved the quality of our manuscript. The authors were partially supported by the NSF Grant \#DMS-1928930, while the authors were in residence at the Simons Laufer Mathematical Sciences Institute for the Quantum Symmetries Program Reunion in Summer 2024. CW was also supported by the NSF Grant \#DMS--2348833, and an AMS Claytor-Gilmer Research Fellowship. HY was supported by a start-up grant from the University of Alberta and an NSERC Discovery Grant.

\bibliography{Nondeg-mod-cats}
\bibliographystyle{amsrefs}

\end{document}